\documentclass[a4paper, leqno, twoside]{amsart}
\usepackage[utf8]{inputenc}
\usepackage{esint}
\usepackage{amssymb,amsmath,mathtools,latexsym}
\usepackage{amsthm}
\usepackage[usenames,dvipsnames]{color}
\usepackage{marginnote}

\usepackage[alphabetic]{amsrefs}
\usepackage{setspace}
\usepackage{graphicx}
\usepackage{mathrsfs}
\usepackage{mdwlist}
\usepackage{url}
\usepackage{bbm}
\usepackage[hidelinks]{hyperref}
\usepackage{breakurl}
\usepackage[all,cmtip]{xy}
\usepackage{ifthen}

\newcommand{\set}[2]{\left\{ {#1} \,;\, {#2} \right\}}
\newcommand{\Ltwo}{L^{2}}
\newcommand{\smooth}{C^{\infty}}
\newcommand{\G}{\mathbf{G}} 
\newcommand{\one}{\mathbbm{1}} 
\newcommand{\D}{\mathbb{D}}
\newcommand{\field}{\mathfrak{k}} 
\newcommand{\N}{\mathbb{N}}
\newcommand{\Q}{\mathbb{Q}}
\newcommand{\R}{\mathbb{R}}
\newcommand{\Z}{\mathbb{Z}}
\newcommand{\bN}{\mathbb{N}}
\newcommand{\bQ}{\mathbb{Q}}
\newcommand{\bR}{\mathbb{R}}
\newcommand{\bZ}{\mathbb{Z}}
\newcommand\rquot[2]{
  \mathchoice
  {
    \text{\raise0.5ex\hbox{$#1$}\big/\lower0.5ex\hbox{$#2$}}%
  }
  {
    #1\,/\,#2
  }
  {
    #1\,/\,#2
  }
  {
    #1\,/\,#2
  }
}

\newcommand\lquot[2]{
  \mathchoice
  {
    \text{\lower0.5ex\hbox{$#1$}\big\backslash\raise0.5ex\hbox{$#2$}}%
  }
  {
    #1\,\backslash\,#2
  }
  {
    #1\,\backslash\,#2
  }
  {
    #1\,\backslash\,#2
  }
}

\newcommand\lrquot[3]{
  \mathchoice
  {
    \text{\lower0.5ex\hbox{$#1$}\big\backslash\raise0.5ex\hbox{$#2$}\big/
      \lower0.5ex\hbox{$#3$}}%
  }
  {
    #1\,\backslash\,#2\,/\,#3
  }
  {
    #1\,\backslash\,#2\,/\,#3
  }
  {
    #1\,\backslash\,#2\,/\,#3
  }
}
\newcommand{\weakstar}{weak$^{\ast}$~}
\newcommand{\Dcal}{\mathcal{D}}
\newcommand{\Jcal}{\mathcal{J}}

\newcommand{\Pcal}{\mathcal{P}}
\newcommand{\Scal}{\mathcal{S}}
\newcommand{\der}[1][]{
  \ifthenelse{ \equal{#1}{} }%
  {\ensuremath{\mathrm{d}}}%
  {\ensuremath{\mathrm{d}_{#1}\!}}%
}
\newcommand{\vol}[1]{\mathrm{Vol}( {#1} )}
\newcommand{\covol}[1]{\mathrm{Covol}\left( {#1} \right)}
\newcommand{\liea}{\mathfrak{a}}
\newcommand{\lieg}{\mathfrak{g}}
\newcommand{\liem}{\mathfrak{m}}
\newcommand{\liep}{\mathfrak{p}}
\newcommand{\liesltwo}{\mathfrak{sl}_{2}}
\newcommand{\SO}{\mathrm{SO}}
\newcommand{\SL}{\mathrm{SL}}
\newcommand{\Rspan}{\mathrm{span}_{\R}}
\newcommand{\cl}[1]{\overline{ {#1} }}
\newtheorem{theorem}{Theorem}\numberwithin{theorem}{section}
\newtheorem{lemma}[theorem]{Lemma}
\newtheorem{proposition}[theorem]{Proposition}
\newtheorem{corollary}[theorem]{Corollary}
\theoremstyle{remark}\newtheorem{remark}[theorem]{Remark}
\theoremstyle{definition}\newtheorem{definition}[theorem]{Definition}
\begin{document}
\title[Primitive rational points]{Primitive rational points on expanding horospheres in Hilbert modular surfaces}
\author{Manuel Luethi}
\address{Department of Mathematics, ETH Zurich, R\"amistrasse 101, 8092 Z\"urich, Switzerland}
\email{manuel.luethi@math.ethz.ch}
\thanks{M.~L.~acknowledges the support by the SNSF (Grant 200021-178958)}
\date{\today}
\begin{abstract}
In recent work by Einsiedler, Mozes, Shah and Shapira the limiting distributions of primitive rational points on expanding horospheres was examined in arbitrary dimension, and a suspended version of this result was announced. Motivated by this, we prove an analog to the announcement for primitive rational points in Hilbert modular surfaces via effective mixing.
\end{abstract}
\maketitle
\allowdisplaybreaks\frenchspacing

\section{Introduction}\label{sec:introduction}
This article grew out of a well-known observation on the relation between Kloosterman sums and~$\mathrm{SL}_{2}(\mathbb{Z})\backslash\mathrm{SL}_{2}(\mathbb{R})$ (e.g. \cites{Marklof2010,Primitive}). In order to motivate the problem examined, we will first give a description of this simpler case and then reformulate it for the situation under consideration.
\subsection{Motivation}\label{sec:motivation}
We consider the natural action of~$\mathrm{SL}_{2}(\mathbb{R})$ on the unit tangent bundle to the modular surface~$\mathrm{SL}_{2}(\mathbb{Z})\backslash\mathrm{SL}_{2}(\mathbb{R})$. For what follows, we denote
\begin{align*}
  A&=\left\{a(y)=\begin{pmatrix}y^{-1} & 0 \\ 0 & y\end{pmatrix}:y\in\mathbb{R}_{>0}\right\},\\
  U&=\left\{u_{t}=\begin{pmatrix}1 & t \\ 0 & 1\end{pmatrix}:t\in\mathbb{R}\right\},\text{ and }\\
  V&=\left\{v_{s}=\begin{pmatrix}1 & 0 \\ s & 1\end{pmatrix}:s\in\mathbb{R}\right\}.
\end{align*}
The action of~$A$ on~$\mathrm{SL}_{2}(\mathbb{Z})\backslash\mathrm{SL}_{2}(\mathbb{R})$ is a reparametrized realization of the geodesic flow. It has been shown by Sarnak \cite{Sarnak1981}, that long horocycles equidistribute in~$\mathrm{SL}_{2}(\mathbb{Z})\backslash\mathrm{SL}_{2}(\mathbb{R})$, i.e.~let~$\mu_{y}$ denote the push-forward of the normalized~$U$-invariant measure on~$\mathrm{SL}_{2}(\mathbb{Z})U$ under the action of the element~$a(y)$, or equivalently the unique~$U$-invariant probability measure on~$\Gamma Ua(y)$, then~$\mu_{y}$ equidistributes towards the unique~$\mathrm{SL}_{2}(\mathbb{R})$-invariant probability measure on~$\mathrm{SL}_{2}(\mathbb{Z})\backslash\mathrm{SL}_{2}(\mathbb{R})$---the Haar measure---as~$y\to\infty$. On the other hand, the interpretation of~$\mathrm{SL}_{2}(\mathbb{Z})\backslash\mathrm{SL}_{2}(\mathbb{R})$ as the space of unimodular lattices in~$\mathbb{R}^{2}$ combined with Mahler's compactness criterion yields that~$\Gamma u_{t}a(y)$ diverges to infinity if and only if~$t$ is a rational number, and similarly for~$\Gamma v_{s}a(y)^{-1}$. As the rational points~$\set{\Gamma u_{t}a(y)}{t\in\mathbb{Q}}$ form a dense subset of an equidistributing orbit, it is clear that the divergence is not uniform, and one might ask about the speed of divergence. The following intriguing arithmetic miracle answers this question completely:
\begin{proposition}\label{prop:arithmeticmiracle}
  The intersection~$\Gamma Ua(y)\cap\Gamma V$ is non-empty if and only if~$y=n$ for some~$n\in\mathbb{N}$. If~$t,s\in\mathbb{R}$ satisfy~$\Gamma u_{t}a(n)=\Gamma v_{s}$, then there is~$j\in\mathbb{Z}$ coprime to~$n$ such that~$t=\frac{j}{n}$ and~$s=\frac{j^{\times}}{n}$ for~$j^{\times}j\equiv 1\,\mathrm{mod}\,n$. Conversely~$\Gamma u_{j/n}a(n)\in\Gamma V$ whenever~$j$ and~$n$ are coprime.
\end{proposition}
The proof of Proposition \ref{prop:arithmeticmiracle} is an elementary calculation. Interpreting~$a(y)$ as the geodesic flow for time~$\log y$, and using the uniform divergence of~$\Gamma Va(y)$ as~$y\to\infty$, we understand that the points~$\set{\Gamma u_{j/n}}{\gcd(j,n)=1}$ under right multiplication with~$a(y)$ uniformly diverge into the cusp after time~$\log y\geq\log n$. It remains to examine the behaviour of these points for times~$\alpha\log n$ for~$\alpha$ in the interval~$(0,1)$, i.e.~the behaviour of the sets
\begin{equation*}
  \set{\Gamma u_{j/n}}{\gcd(j,n)=1}a(n^{\alpha}).
\end{equation*}
Note that Proposition \ref{prop:arithmeticmiracle} implies
\begin{equation*}
  \set{\Gamma u_{j/n}}{\gcd(j,n)=1}a(n^{\alpha})=\set{\Gamma v_{j/n}}{\gcd(j,n)}a(n^{\alpha-1}).
\end{equation*}
As the difference between~$U$ and~$V$ orbits only lies in the choice of the forward direction for the geodesic flow, it suffices to consider~$\alpha\in(0,1/2]$. As announced in \cites{Primitive,survey}, for fixed~$\alpha$ one can show that the normalized counting measures on these sets equidistribute towards the Haar measure on~$\mathrm{SL}_{2}(\mathbb{Z})\backslash\mathrm{SL}_{2}(\mathbb{R})$ as~$n\to\infty$. This is examined in greater detail in~\cite{sl2}.
\subsection{Setup}\label{sec:setup}
We want to recast the discussion from above in the context of Hilbert modular surfaces. In what follows,~$\field$ is a totally real number field of degree~$d$ and~$\mathfrak{o}$ is its ring of integers. We will choose some enumeration~$\{\sigma_{i};1\leq i\leq d\}$ of the Galois embeddings~$\sigma_{i}:\field\to\mathbb{R}$. These embeddings induce an embedding of the~$\mathbb{Q}$-vector space~$\field$ in~$\mathbb{R}^{d}$, which is given by sending~$x\in \field$ to the vector~$\sigma x$ whose coordinates are~$(\sigma x)_{i}=\sigma_{i}x$. Similarly, one obtains an embedding of~$\mathrm{SL}_{2}(\field)$ in~$\mathrm{SL}_{2}(\mathbb{R})^{d}$, which sends a matrix~$g\in\mathrm{SL}_{2}(\field)$ to the element, whose~$i$-th component is the image of~$g$ under coordinate-wise application of~$\sigma_{i}$. In what follows, we will write~$\Gamma$ for the image of~$\mathrm{SL}_{2}(\mathfrak{o})$ in~$G=\mathrm{SL}_{2}(\mathbb{R})^{d}$. Applying the restriction of scalars functor for the~$\field$-group~$\mathrm{SL}_{2}$ with respect to the subfield~$\mathbb{Q}$ yields a semisimple~$\mathbb{Q}$-group~$\G$, and one can show that there is an isomorphism~$\G(\mathbb{R})\cong\mathrm{SL}_{2}(\mathbb{R})^{d}$ that restricts to an isomorphism~$\G(\mathbb{Z})\cong\Gamma$. In particular, it follows that~$\Gamma\leq G$ is an irreducible congruence lattice (cf.~\cite{Tomanov2002}). In what follows, given~$s,t\in\mathbb{R}^{d}$, we denote by~$u_{t}\in G$ the element whose~$i$-th coordinate is the matrix~$(\begin{smallmatrix}1 & t_{i} \\ 0 & 1\end{smallmatrix})$, by~$v_{s}\in G$ the element whose~$i$-th coordiante is the matrix~$(\begin{smallmatrix}1 & 0 \\ s_{i} & 1\end{smallmatrix})$, and for~$y\in(\mathbb{R}\setminus\{0\})^{d}$ by~$a(y)$ the element whose~$i$-th coordinate is given by the matrix~$\big(\begin{smallmatrix}y_{i}^{-1} & 0 \\ 0 & y_{i}\end{smallmatrix}\big)$. We define the subgroups
\begin{align*}
  A&=\set{a(y)}{y\in(\mathbb{R}\setminus\{0\})^{d}}\\
  U&=\set{u_{t}}{t\in\mathbb{R}^{d}}\\
  V&=\set{v_{s}}{s\in\mathbb{R}^{d}}
\end{align*}
Given~$x\in \field^{\times}$, we say that~$x$ is totally positive if~$\sigma_{i}x$ is positive for all~$i$. Note that the totally positive elements form a subgroup of finite index. Given a totally positive element~$x\in \field^{\times}$ and~$\alpha\in(0,1)$, we will denote by~$a_{\alpha}(x)$ the matrix~$a(x^{\alpha})$, where~$x^{\alpha}\in\mathbb{R}^{d}$ is the vector with entries~$(x^{\alpha})_{i}=(\sigma_{i}x)^{\alpha}$. Given two vectors~$t,s\in\R^{d}$, we let~$ts\in\R^{d}$ denote the vector satisfying~$(ts)_{i}=t_{i}s_{i}$ ($1\leq i\leq d$). In this way, given~$x,y\in \field$, we have~$\sigma(xy)=(\sigma x)(\sigma y)$ and for~$x,y\in \field^{\times}$ we get~$a(xy)=a(x)a(y)$.

In this setup the orbit~$\Gamma Ua_{T}$ becomes equidistributed as~$T\to\infty$ if~$a_{T}$ is the matrix in~$A$ defined by the vector whose entries are all equal to~$T$. Repeating the question from Section \ref{sec:motivation} and the calculation required for Proposition \ref{prop:arithmeticmiracle}, one obtains
\begin{proposition}\label{prop:arithmeticmiraclenumberfield}
  Let~$t,s\in\mathbb{R}^{d}$ and~$y\in(\mathbb{R}\setminus\{0\})^{d}$. The following are equivalent:
  \begin{enumerate}
  \item~$\Gamma u_{t}a(y)=\Gamma v_{s}$.
  \item~$y$,~$yt$ and~$ys$ are integral in~$\field$ satisfying~$(yt)(ys)\equiv 1\,\mathrm{mod}\,y$.
  \end{enumerate}
\end{proposition}
Again we will examine the behavior of the primitive rational points, i.e.~the sets
\begin{equation*}
  \mathcal{P}_{y}^{\alpha,\times}=\set{\Gamma u_{j/y}}{j\in\big(\mathfrak{o}/y\mathfrak{o}\big)^{\times}}a_{\alpha}(y)
\end{equation*}
for---for the sake of simplicity of notation---totally positive elements~$y\in\mathfrak{o}$. In this case the ideal~$y\mathfrak{o}$ is a finite index subgroup in~$\mathfrak{o}$ of index~$N(y)$, where~$N(y)$ is the product of the images of~$y$ under the distinct Galois embeddings. We denote~$\phi(y)=\lvert(\rquot{\mathfrak{o}}{y\mathfrak{o}})^{\times}\rvert$. We will prove effective equidistribution of the sets~$\mathcal{P}_{y}^{\alpha,\times}$ for the unique~$G$-invariant probability measure on~$\Gamma\backslash G$ as~$N(y)\to\infty$. More precisely, we will prove
\begin{theorem}\label{thm:mainthmnumberfield}
  Let~$\field$ be a totally real number field of degree~$d>1$,~$\alpha\in(0,1)$. There exists an~$\Ltwo$-Sobolev norm~$\mathcal{S}$ on~$C_{c}^{\infty}(\Gamma\backslash G)$ such that for all totally positive~$y\in\mathfrak{o}$ 
  \begin{equation*}
    \bigg\lvert\frac{1}{\phi(y)}\sum_{j\in(\mathfrak{o}/y\mathfrak{o})^{\times}}f\big(\Gamma u_{j/y}a_{\alpha}(y)\big)-\int_{\Gamma\backslash G}f\bigg\rvert\ll\alpha^{-\frac{1}{2}}\Big(\tfrac{(\log\log N(y))^{d}}{\log N(y)}\Big)^{\frac{1}{2}}\mathcal{S}(f)
  \end{equation*}
  for all~$f\in C_{c}^{\infty}(\Gamma\backslash G)$, with implicit constant independent of~$y,\alpha$ and~$f$. In particular, the primitive rational points of denominator~$y$ on~$\Gamma Ua_{\alpha}(y)$ equidistribute in~$\Gamma\backslash G$ as~$N(y)\to\infty$.
\end{theorem}
\begin{remark}
  It is worthwhile pointing out that the implicit constant in Theorem~\ref{thm:mainthmnumberfield} depends on the spectral gap of the quotient~$\lquot{\Gamma}{G}$. As the spectral gap goes to zero, the implicit constant diverges to infinity.
\end{remark}
Let us give an ineffective outline of the proof. For what follows, assume that the rational points
\begin{equation*}
  \mathcal{P}_{y}^{\alpha}=\set{\gamma u_{j/y}}{j\in\mathfrak{o}}a_{\alpha}(y)\subseteq\lquot{\Gamma}{G}
\end{equation*}
equidistribute as~$\lvert N(y)\rvert\to\infty$, i.e.~the measures
\begin{equation}\label{eq:definitioncountingmeasurerationalpoints}
  \mu_{y}^{\alpha}=\frac{1}{\lvert N(y)\rvert}\sum_{j\in\rquot{\mathfrak{o}}{y\mathfrak{o}}}\delta_{\Gamma u_{j/y}a_{\alpha}(y)}
\end{equation}
converge to the unique~$G$-invariant probability measure~$m_{\lquot{\Gamma}{G}}$ on~$\lquot{\Gamma}{G}$ in the \weakstar topology. Let us also denote
\begin{equation}\label{eq:definitioncountingmeasureprimitiverationalpoints}
  \mu_{y}^{\alpha,\times}=\frac{1}{\phi(y)}\sum_{j\in(\rquot{\mathfrak{o}}{y\mathfrak{o}})^{\times}}\delta_{\Gamma u_{j/y}a_{\alpha}(y)}
\end{equation}

Assume first that~$y\mathfrak{o}\subseteq\mathfrak{o}$ is a prime ideal, so that
\begin{equation*}
  \mu_{y}^{\alpha,\times}=\big(1+\tfrac{1}{\phi(y)}\big)\mu_{y}^{\alpha}-\tfrac{1}{\phi(y)}\delta_{\Gamma a_{\alpha}(y)}.
\end{equation*}
Equidistribution of rational points immediately implies that~$\mu_{y}^{\alpha,\times}$ converges to the unique~$G$-invariant probability measure as~$\lvert N(y)\rvert\to\infty$. More generally, let~$\varepsilon>0$ and define
\begin{equation*}
  \D(\varepsilon)=\set{y\in\mathfrak{o}}{\phi(y)\geq\varepsilon\lvert N(y)\rvert}.
\end{equation*}
From this one can relatively easily deduce that for sequences~$(y_{k})_{k\in\N}$ in~$\D(\varepsilon)$ satisfying~$\lvert N(y_{k})\rvert\to\infty$, the primitive rational points
\begin{equation*}
  \mathcal{P}_{y}^{\alpha,\times}=\set{\Gamma u_{y^{-1}j}a_{\alpha}(y)}{j\in(\mathfrak{o}/y\mathfrak{o})^{\times}}
\end{equation*}
equidistribute towards the unique~$G$-invariant probability measure on~$\Gamma\backslash G$. To this end let~$\mathcal{P}_{y}^{\alpha,0}=\mathcal{P}_{y}^{\alpha}\setminus\mathcal{P}_{y}^{\alpha,\times}$ and denote by~$\mu_{y}^{\alpha,\times}$ and~$\mu_{y}^{\alpha,0}$ the normalized counting measures on~$\Pcal_{y}^{\alpha,\times}$ and~$\Pcal_{y}^{\alpha,0}$ respectively. Then
\begin{equation*}
  \mu_{y}^{\alpha}=\tfrac{\phi(y)}{\lvert N(y)\rvert}\mu_{y}^{\alpha,\times}+\tfrac{\lvert N(y)\rvert-\phi(y)}{\lvert N(y)\rvert}\mu_{y}^{\alpha,0}
\end{equation*}
Let~$(y_{k})_{k\in\N}$ be a sequence in~$\D(\varepsilon)$ satisfying~$\lvert N(y_{k})\rvert\to\infty$. After possibly restricting to a subsequence, we can assume that~$\frac{\phi(y_{k})}{\lvert N(y_{k})\rvert}\to\lambda\in[\varepsilon,1]$. As~$\mu_{y_{k}}^{\alpha}$ converges to the~$G$-invariant probability measure on~$\Gamma\backslash G$, the sequence~$\mu_{y_{k}}^{\alpha,\times}$ converges to a probability measure~$\nu$ on~$\Gamma\backslash G$. Dirichlet's unit theorem now yields an element~$g\in G$ which acts ergodically on~$\Gamma\backslash G$ and for which all~$\mu_{y_{k}}^{\alpha,\times}$ are invariant. In particular,~$\nu$ exhibits the same invariance and thus extremality of ergodic measures implies that~$\nu$ equals the~$G$-invariant probability measure, i.e.~we obtain the desired equidistribution. For the general case, i.e.~without assuming that~$y\in\D(\varepsilon)$, an effective argument making use of the spectral gap is required.
\subsection{Organization of the paper}
Section \ref{sec:discrepancytrick} deduces an effective version of von Neumann's ergodic theorem for effectively mixing dynamical systems, which serves as a motivation for later arguments. This section is kept very general. In Section \ref{sec:generalresults} we discuss~$\Ltwo$-Sobolev norms as well as equidistribution of large horospheres and a general bound for the error of approximation of the space average by a sparse subset of a unipotent orbit exhibiting invariance. In Section \ref{sec:ringofintegers}, we first review prime factorization in Dedekind domains and examine the totient function for number fields. Afterwards we use Dirichlet's unit theorem to choose a Cartan subgroup of~$G$ and prove effective equidistribution of large horospheres. In Section \ref{sec:rationalpoints} we deduce equidistribution of rational points on large horospheres. In Section \ref{sec:primitiverationalpoints} we combine the equidistribution of the rational points with a version of the discrepancy trick introduced in Section \ref{sec:discrepancytrick} to prove Theorem \ref{thm:mainthmnumberfield}.
\subsection{Acknowledgements} The author would like to thank Manfred Einsiedler for suggesting the problem and many helpful discussions on the techniques applied. The author would furthermore like to thank Menny Akka, Manfred Einsiedler, Alex Gorodnik, \c{C}a\u{g}r{\i} Sert and Andreas Wieser for comments on an earlier draft.
\section{Effective mixing and von Neumann's ergodic theorem}\label{sec:discrepancytrick}
In this section we prove a tool---the \emph{discrepancy trick}---that we will use throughout the remainder of the article. It is certainly well-known to experts. It can be summarized as follows: Effective mixing for an invertible dynamical system implies the mean ergodic theorem with a rate. As its application is not restricted to the topic of this article, we state a general version. In what follows, we assume that~$(X,\mathcal{B},\mu,T)$ is an invertible probability measure preserving system and that~$\mathcal{A}\subseteq \Ltwo(X,\mu)$ is a set of real-valued functions. Given~$f\in \Ltwo(X,\mu)$, we write
\begin{equation*}
  E_{f}=\int_{X}f(x)\der\mu(x).
\end{equation*}
Assume that~$\psi:\R\to\R_{+}$ is a bounded, symmetric function decreasing monotonically on the positive half line. Assume that~$T$ is an effectively mixing transformation with rate~$\psi$, i.e.~for all~$f,g\in\mathcal{A}$ we have
\begin{equation*}
  \big\lvert\langle T^{k}f,g\rangle-E_{f}E_{g}\big\rvert\ll\psi(k)\mathcal{S}(f)\mathcal{S}(g),
\end{equation*}
where~$\mathcal{S}$ is a norm on~$\mathcal{A}$ dominating the~$\Ltwo(X,\mu)$-norm. Given~$K\in\N$ and~$f\in \Ltwo(X,\mu)$, let~$A_{K}(f)=\frac{1}{K}\sum_{n=0}^{K-1}f\circ T^{n}$.
\begin{proposition}\label{prop:effectivevonneumann}
  Let~$f\in\mathcal{A}$, then for all~$\varsigma\in(0,1)$ we have
  \begin{equation*}
    \lVert A_{K}(f)-E_{f}\rVert_{2}^{2}\ll \big(K^{-\varsigma}+\psi(K^{1-\varsigma})\big)\Scal(f)^{2}.
  \end{equation*}
\end{proposition}
\begin{proof}
  As~$T$ is invertible and~$f$ is real-valued, we have
  \begin{align*}
    \lVert A_{K}(f)-E_{f}\rVert_{2}^{2}&\ll\frac{1}{K^{2}}\sum_{m,n=0}^{K-1}\big\langle T^{m-n}(f-E_{f}),f-E_{f}\big\rangle\\
    &\ll\frac{1}{K^{2}}\sum_{0\leq n\leq m<K}\big\langle T^{m-n}(f-E_{f}),f-E_{f}\big\rangle\\
   &\ll\frac{\mathcal{S}(f)^{2}}{K^{2}}\sum_{0\leq n\leq m<K}\psi(m-n).
  \end{align*}
  We write the range of summation as~$P_{K}\sqcup Q_{K}$, where
  \begin{align*}
    P_{K}&=\{(m,n)\in\N_{0}:m,n<K,\lvert m-n\rvert\leq K^{1-\varsigma}\}\\
    Q_{K}&=\{(m,n)\in\N_{0}:m,n<K\}\setminus P_{K}.
  \end{align*}
  The cardinality of~$P_{K}$ equals
  \begin{equation*}
    \lvert P_{K}\rvert=2K(\lfloor K^{1-\varsigma}\rfloor+1)-2\tfrac{K(K-1)}{2}.
  \end{equation*}
  Hence
  \begin{align*}
    \sum_{0\leq n\leq m<K}\psi(m-n)&=\sum_{(m,n)\in P_{K}}\psi(m-n)+\sum_{(m,n)\in Q_{K}}\psi(m-n)\\
    &\ll K^{2-\varsigma}+K(2K-2\lfloor K^{1-\varsigma}\rfloor-3)\psi(K^{1-\varsigma}).
  \end{align*}
  After division by~$K^{2}$, the claim follows.
\end{proof}
\begin{remark}
  As will become apparent in Section \ref{sec:primitiverationalpoints}, the argument provided is quite wasteful and if the function~$\psi$ is well-understood, it is often possible to deduce much better bounds.
\end{remark}
\section{Notation and general results}\label{sec:generalresults}
In this section we will introduce the general notation used throughout the article and a few general results, which are valid independent of the specific context of this paper. We will consider the set of real points~$G$ of a semisimple, linear~$\mathbb{Q}$-group~$\G$, and assume that~$G$ does not have any compact factors. We will once and for all fix a faithful rational representation~$\pi:\G\to\mathrm{GL}_{N}$ (cf.~\cite[Thm.~2.3.7]{SpringerLinearGroups}). We equip~$\mathbb{R}^{N}$ with a fixed Euclidean structure and~$\mathrm{GL}_{N}(\mathbb{R})$ with some compatible operator norm~$\lVert\cdot\rVert_{\pi}$. For~$g\in G$, we let
\begin{equation*}
  \lVert g\rVert=\max\big\{\lVert\pi(g)\rVert_{\pi},\lVert\pi(g)^{-1}\rVert_{\pi}\big\}.
\end{equation*}
We assume that~$\Gamma\leq G$ is an irreducible arithmetic lattice, and in particular that every element~$g\in G$ not contained in a compact subgroup acts ergodically on~$\lquot{\Gamma}{G}$ with respect to the unique~$G$-invariant probability measure~$m_{\lquot{\Gamma}{G}}$ on~$\lquot{\Gamma}{G}$. We will usually write
\begin{equation*}
  m_{\lquot{\Gamma}{G}}(f)=\int_{\lquot{\Gamma}{G}}f(\Gamma g)\der\Gamma g\qquad\big(f\in C_{c}(\lquot{\Gamma}{G})\big).
\end{equation*}
Denote by~$p:G\to\Gamma\backslash G$ the canonical projection. Given an element~$g\in G$, we denote by~$l_{g},r_{g}:G\to G$ the diffeomorphisms given by left- and right-multiplication with~$g$ respectively.
We denote by~$\mathfrak{g}$ the Lie algebra of~$G$, i.e.~the tangent space at the identity. Note that~$\mathfrak{g}$ carries the Euclidean structure inherited from~$\mathbb{R}^{N^{2}}$. The derivative~$D(g)$ of the local diffeomorphism~$p\circ l_{g}$ at the identity maps any basis of~$\mathfrak{g}$ to a basis of the tangent space at~$\Gamma g$ and its image is independent of the representative of~$\Gamma g$. For any~$X\in\mathfrak{g}$, the map~$g\mapsto D(g)X$ is smooth, thus any choice of a basis~$\mathcal{B}$ yields a smooth frame bundle on~$\Gamma\backslash G$. 

The Hilbert space~$\Ltwo(\Gamma\backslash G)$ yields a unitary representation of~$G$, where for~$f\in \Ltwo(\Gamma\backslash G)$ and~$g\in G$ the element~$g\cdot f\in \Ltwo(\Gamma\backslash G)$ is defined by~$g\cdot f(x)=f(xg)$ almost everywhere. Given~$f\in \Ltwo(\Gamma\backslash G)$ and~$\varphi\in C_{c}(G)$, we denote by~$\varphi\star f$ the convolution of~$\varphi$ with~$f$, i.e.\
\begin{equation*}
  \varphi\star f(x)=\int_{G}\varphi(g)(g\cdot f)(x)\der g,
\end{equation*}
where~$\der g$ denotes integration with respect to the Haar measure on~$G$. Using dominated convergence, one can show that~$\varphi\star f$ is smooth and for any left invariant vector field~$X\in\mathfrak{g}$ we have~$D(g)X(\varphi\star f)=X(\varphi)\star f(\Gamma g)$.

Using the mean-value theorem, one notes that the error of approximation of the function~$f$ at a point by its average on a small ball around that point depends on the smoothness properties of the specific function~$f$. In a similar manner, any effective equidistribution statements examined in this project will depend on the smoothness properties of the test function. The appropriate tool to measure these are~$\Ltwo$-Sobolev norms. 
\subsection{$\Ltwo$-Sobolev norms and approximate identities}\label{sec:sobolevnorms}
In this section, we introduce~$\Ltwo$-Sobolev norms and discuss some of their properties. The use of Sobolev norms in this area has become quite standard, and hence we will not prove all the properties used. For more detailed discussions of Sobolev norms on homogeneous spaces, we refer to \cites{Venkatesh2010,EMV}. The following construction works for general discrete subgroups. Let~$\mathrm{ht}:\Gamma\backslash G\to(0,\infty)$ be a smooth function. For our purposes, an~$\Ltwo$-Sobolev norm of degree~$\ell$ on~$\Gamma\backslash G$ with height function~$\mathrm{ht}:\lquot{\Gamma}{G}\to(0,\infty)$ is a choice of a frame bundle~$\mathcal{B}$ as above, together with the map~$\mathcal{S}:C_{c}^{\infty}(\Gamma\backslash G)\to\mathbb{R}$ given by
\begin{equation*}
  \mathcal{S}(f)^{2}=\sum_{X\in\Dcal_{\ell}(\mathcal{B})}\lVert(1+\mathrm{ht})^{\ell}X(f)\rVert_{2}^{2}\quad\big(f\in C_{c}^{\infty}(\Gamma\backslash G)\big),
\end{equation*}
where~$\Dcal_{\ell}(\mathcal{B})$ is the collection of vector fields given as monomials in the elements of~$\mathcal{B}$ of degree at most~$\ell$. One easily sees that the map~$\mathcal{S}$ is a norm defined on a dense subset of~$\Ltwo(\Gamma\backslash G)$. On~$G$ we choose for~$\mathrm{ht}$ a constant function whereas in the case of a lattice,~$\mathrm{ht}$ will be the height function as given in \cite[p.~153]{EMV}. In particular, it holds that~$1\ll\mathrm{ht}$ and for all~$g\in G$ and~$x\in\Gamma\backslash G$
  \begin{equation}\label{eq:multiplicativityheight}
    \mathrm{ht}(xg)\ll\lVert g\rVert\mathrm{ht}(x).
  \end{equation}
  The choice the height function as in \cite{EMV} guarantees validity of the Sobolev embedding theorem, i.e.~there exists some~$\ell_{0}$ such that whenever~$\mathcal{S}$ has degree at least~$\ell_{0}$, then 
  \begin{equation}\label{eq:sobolevembedding}
    \lVert f\rVert_{\infty}\ll\mathcal{S}(f)
  \end{equation}
  for all~$f\in C_{c}^{\infty}(\Gamma\backslash G)$. Here the implicit constant depends on the choice of~$\mathcal{S}$. In what follows, given an~$\Ltwo$-Sobolev norm~$\mathcal{S}$, we will tacitly assume that it has sufficiently large degree for \eqref{eq:sobolevembedding} to hold.

Finally, we note that the~$\Ltwo$-Sobolev norms do not depend on the initial choice of the basis of~$\mathfrak{g}$ in the sense that for any two choices the resulting norms are equivalent.
\begin{lemma}\label{lem:productofsobolevnorms}
  Let~$\mathcal{S}$ be an~$\Ltwo$-Sobolev norm. Then there are an~$\Ltwo$-Sobolev norm~$\mathcal{S}^{\prime}$ of possibly higher degree and a constant~$C_{\ell}$ depending solely on the degree~$\ell$ of~$\mathcal{S}$ such that for all~$f_{1},f_{2}\in C_{c}^{\infty}(\Gamma\backslash G)$ we have
  \begin{equation*}
    \mathcal{S}(f_{1}f_{2})\leq C_{\ell}\mathcal{S}^{\prime}(f_{1})\mathcal{S}^{\prime}(f_{2}).
  \end{equation*}
\end{lemma}
\begin{proof}
  We expand the summands using the triangle inequality for the norm~$\lVert\cdot\rVert_{2}$ and use the fact that Lie algebra elements act as derivations in order to obtain a combinatorially defined sum of norms of~$(1+\mathrm{ht})^{\ell}A(f_{1})B(f_{2})$, where~$A,B$ are monomials of degree at most~$\ell$ in elements from~$\mathcal{B}$. Using \eqref{eq:sobolevembedding} we get
  \begin{align*}
    \lVert(1+\mathrm{ht})^{\ell}A(f_{1})B(f_{2})\rVert_{2}&\leq\lVert(1+\mathrm{ht})^{\ell}A(f_{1})\rVert_{2}\lVert B(f_{2})\rVert_{\infty}\\
    &\ll\lVert(1+\mathrm{ht})^{\ell}A(f_{1})\rVert_{2}\mathcal{S}(B(f_{2})).
  \end{align*}
  Clearly~$\mathcal{S}(B(f_{2}))\ll \mathcal{S}^{\prime}(f_{2})$ and~$\lVert(1+\mathrm{ht})^{\ell}A(f_{1})\rVert_{2}\ll\mathcal{S}^{\prime}(f_{1})$, whenever~$\mathcal{S}^{\prime}$ is an~$\Ltwo$-Sobolev norm of degree at least~$2\ell$.
\end{proof}
\begin{lemma}\label{lem:shiftedsobolevnorms}
  Let~$\mathcal{S}$ be an~$\Ltwo$-Sobolev norm of degree~$\ell$ on~$C_{c}^{\infty}(\Gamma\backslash G)$. Let~$f\in C_{c}^{\infty}(\Gamma\backslash G)$ and~$g\in G$ arbitrary. Then
  \begin{equation*}
    \mathcal{S}(g\cdot f)\ll\lVert g\rVert^{2\ell}\mathcal{S}(f).
  \end{equation*}
\end{lemma}
\begin{proof}
  Given~$X\in\mathfrak{g}$, let~$\overline{X}$ be the vector field defined by~$\overline{X}_{\Gamma h}(f)=D(h)X(f)$ as discussed previously. One calculates for~$g\in G$ and~$f\in C_{c}^{\infty}(\Gamma\backslash G)$
  \begin{equation*}
    \lVert(1+\mathrm{ht})^{\ell}\overline{X}(g\cdot f)\rVert_{2}=\lVert(1+g^{-1}\cdot\mathrm{ht})^{\ell}\overline{\mathrm{Ad}_{g^{-1}}X}(f)\rVert_{2}\ll\lVert g\rVert^{2\ell}\mathcal{S}(f),
  \end{equation*}
  and the general statement follows by iteration of this argument.
\end{proof}
The final general property we want to state is a form of a Lipshitz bound. We refer the reader to~\cite[\textsection3.7]{EMV} for an outline of the proof.
\begin{lemma}\label{lem:Lipshitzbound}
  There exists some~$\ell_{0}\in\bN$ such that for all~$\Ltwo$-Sobolev norms~$\mathcal{S}$ of degree at least~$\ell_{0}$ the following is true. Let~$d$ denote a left-invariant metric on~$G$. For all~$g_{1},g_{2}\in G$ and~$f\in C_{c}^{\infty}(\Gamma\backslash G)$ one has
  \begin{equation*}
    \lVert g_{1}f-g_{2}f\rVert_{\infty}\ll d(g_{1},g_{2})\mathcal{S}(f).
  \end{equation*}
\end{lemma}

Another ingredient we use throughout this article are \emph{approximate identities}. An approximate identity on a Lie group~$G$ equipped with a Riemannian metric is a family of non-negative, smooth functions~$\{\varphi_{\varepsilon};\varepsilon\in(0,\varepsilon_{0})\}$ for some~$\varepsilon_{0}>0$, such that for all~$\varepsilon$ the following hold:
\begin{itemize}
\item~$\varphi_{\varepsilon}$ has support contained in the~$\varepsilon$-ball~$B_{\varepsilon}^{G}\subseteq G$ around the identity,
\item~$\varphi_{\varepsilon}$ is symmetric, i.e.~for all~$g\in G$ holds~$\varphi_{\varepsilon}(g^{-1})=\varphi_{\varepsilon}(g)$,
\item~$\int_{G}\varphi_{\varepsilon}=1$.
\end{itemize}
We will also speak of single functions~$\varphi_{\varepsilon}$ as approximate identities, by which we mean that~$\varphi_{\varepsilon}$ is a member of a family of functions defined as above. Construction of approximate identities on Lie groups is elementary: one defines approximate identities on~$\mathfrak{g}$ and uses the exponential map to descend to functions on~$G$ for~$\varepsilon$ sufficiently small. The diffeomorphism between a neighbourhood of the origin in~$\mathfrak{g}$ and the identity in~$G$ allow us to express the Haar measure restricted to sufficiently small neighbourhoods of the identity as integration on~$\mathfrak{g}$ against a smooth density which does not vanish on some neighbourhood of the origin. General continuity arguments then show, that an approximate identity can be found, such that for any~$\Ltwo$-Sobolev norm~$\mathcal{S}$ on~$G$ we have
\begin{equation*}
  \mathcal{S}(\varphi_{\varepsilon})\ll\varepsilon^{-\frac{1}{2}\dim(G)-\ell},
\end{equation*}
where~$\ell$ is the degree of~$\mathcal{S}$.
\subsection{Equidistribution of expanding horospheres}\label{sec:horospheres}
The goal is to prove an equidistribution result for discrete subsets of long horospherical orbits, for which we rely on the effective equidistribution of the full orbits. It will turn out that the link between the full orbit and the discrete subset is best described in greater generality than the scope of the article. Given the assumptions from Section \ref{sec:generalresults} on~$G$ and~$\Gamma$, \cite[Lem.~3]{Bekka} and \cite[\textsection2.4.4]{KleinbockMargulisPieces} implies that for any two smooth functions~$f_{1},f_{2}\in \Ltwo(\Gamma\backslash G)$ we have
\begin{equation}\label{eq:matrixcoefficients}
  \lvert\langle g\cdot f_{1},f_{2}\rangle-E_{f_{1}}\overline{E_{f_{2}}}\rvert\ll\Xi(g)^{\kappa_{\tau}}\mathcal{S}(f_{1})\mathcal{S}(f_{2}),
\end{equation}
where~$\mathcal{S}$ is defined by an~$\Ltwo$-Sobolev norm on~$C_{c}^{\infty}(\Gamma\backslash G)$,~$\Xi:G\to\R$ is the Harish-Chandra spherical function, and~$\kappa_{\tau}>0$ is a fixed constant representing the spectral gap.

Using effective decay of matrix coefficients, we can prove effective equidistribution of large horospheres. The method of proof goes back to the thesis of Margulis and has found many applications, e.g. \cite{EskinMcMullen}. The statement is by no means new \cites{Venkatesh2010,KleinbockMargulisPieces}, however we will depend on the setup of the proof for later discussions, and hence we include it for completeness. Let~$\mathfrak{a}\subseteq\mathfrak{g}$ be a choice of a Cartan subalgebra. It gives rise to a decomposition~$\mathfrak{g}=\bigoplus_{\lambda\in\mathfrak{a}^{\ast}}\mathfrak{g}_{\lambda}$, where
\begin{equation*}
  \mathfrak{g}_{\lambda}=\{v\in\mathfrak{g};\forall H\in\mathfrak{a}: [H,v]=\lambda(H)v\}.
\end{equation*}
One easily checks that~$[\mathfrak{g}_{\lambda},\mathfrak{g}_{\mu}]\subseteq\mathfrak{g}_{\lambda+\mu}$. Similarly, one has~$B(\mathfrak{g}_{\lambda},\mathfrak{g}_{\mu})=0$ whenever~$\lambda\neq-\mu$, where~$B$ is the Killing form on~$\mathfrak{g}$. It follows that~$\mathfrak{g}_{\lambda}$ is non-trivial if and only if~$\mathfrak{g}_{-\lambda}$ is non-trivial. Let~$\Sigma\subseteq\mathfrak{a}^{\ast}$ be the set of roots for~$\mathfrak{a}$, i.e.~the non-trivial~$\lambda\in\mathfrak{a}^{\ast}$ for which~$\mathfrak{g}_{\lambda}$ is non-trivial. Let~$\mathfrak{a}_{+}$ be a choice of a Weyl-chamber and~$\Sigma_{+}$ the corresponding choice of a positive root sytem. Following \cite[Ch.~VII]{Knapp1986}, we let~$\varrho_{+}:\mathfrak{a}\to\R$ denote the functional defined by
\begin{equation*}
  \varrho_{+}(H)=\frac{1}{2}\sum_{\lambda\in\Sigma_{+}}(\dim\mathfrak{g}_{\lambda})\lambda(H)
\end{equation*}
for all $H\in\mathfrak{a}$. We denote~$A_{+}=\exp\mathfrak{a}_{+}$. As argued in \cite[Ch.~VII,~Prop.~7.15]{Knapp1986}, there is some~$\kappa_{H}>0$ such that
\begin{equation}\label{eq:boundharishchandra}
  \Xi(a)\ll e^{-\kappa_{H}\varrho_{+}(\log a)}\qquad(a\in A_{+}).
\end{equation}
The Lie algebra~$\mathfrak{g}$ decomposes as a direct sum~$\mathfrak{p}_{0}\oplus\mathfrak{g}_{-}$ of subalgebras, where
\begin{equation*}
  \mathfrak{g}_{-}=\sum_{\lambda\in\Sigma_{+}}\mathfrak{g}_{-\lambda},\quad\mathfrak{p}_{0}=\mathfrak{g}_{0}\oplus\sum_{\lambda\in\Sigma_{+}}\mathfrak{g}_{\lambda}.
\end{equation*}
Define subgroups~$G_{-}=\exp\mathfrak{g}_{-}$ and~$P_{0}=\exp\mathfrak{p}_{0}$ of~$G$, both normalized by~$A_{+}$. As~$\G$ is linear, both these groups are closed, and we have~$G_{-}\cap P_{0}=\{1\}$. As~$\mathfrak{g}_{-}$ is nilpotent, so is~$G_{-}$. General continuity arguments imply that given any precompact open neighbourhood of the identity in~$G_{-}$, there exists a neighbourhood of the identity in~$P_{0}$, so that the restriction of the multiplication map~$P_{0}\times G_{-}\to G$,~$(b,u)\mapsto bu$ is a diffeomorphism onto an open neighbourhood of the identity in~$G$. Note that the restriction of the Haar measure on~$G$ to~$G_{-}P_{0}$ is given by
\begin{equation*}
  \int_{G} f\propto\int_{G_{-}}\int_{P_{0}}f(ub)\der b\der u,
\end{equation*}
for all integrable~$f$ supported on~$G_{-}P_{0}$, where~$\der b$ and~$\der u$ denote the right Haar measure on~$P_{0}$ and the left Haar measure on~$G_{-}$ respectively. Throughout this article, we will assume that~$\Gamma G_{-}$ is a periodic orbit. Using the assumption, the push-forward~$\mu_{a}$ of the normalized orbit measure on~$\Gamma G_{-}$ under right-multiplication with~$a\in A_{+}$ defines a~$G_{-}$-invariant probability measure on~$\Gamma G_{-}a$.
\begin{proposition}[Equidistribution of long horocycles]\label{prop:banana}
  There exist an~$\Ltwo$-Sobolev norm~$\mathcal{S}$ on~$C_{c}^{\infty}(\Gamma\backslash G)$ and a positive constant~$\kappa>0$, such that for all~$f\in C_{c}^{\infty}(\Gamma\backslash G)$ holds
  \begin{equation*}
    \bigg\lvert\mu_{a}(f)-\int_{\Gamma\backslash G}f\der m_{\lquot{\Gamma}{G}}\bigg\rvert\ll e^{-\kappa\varrho_{+}(\log a)}\mathcal{S}(f).
  \end{equation*}
  The constant~$\kappa$ is determined by the spectral gap of~$\Gamma\backslash G$.
\end{proposition}
\begin{proof}
  As~$G_{-}$ is nilpotent, the orbit~$\Gamma G_{-}$ is compact. Hence there is some~$\varepsilon_{0}$ such that for every~$x\in\Gamma G_{-}$, the map~$B_{2\varepsilon_{0}}^{G}\to\Gamma\backslash G$,~$g\mapsto xg$ is injective, where the metric on~$G$ is assumed to be induced by a left-invariant Riemannian metric. Furthermore, the choice can be made so that the restriction of the left-invariant metric on~$G$ to~$B_{\varepsilon_{0}}^{G_{-}}$ and to~$B_{\varepsilon_{0}}^{P_{0}}$ is Lipschitz-equivalent to the left-invariant metric on these subgroups. After rescaling the metrics on the subgroups~$G_{-}$ and~$P_{0}$, we can assume that for all~$\varepsilon\in(0,\varepsilon_{0})$ we have~$B_{\varepsilon}^{G_{-}}B_{\varepsilon}^{P_{0}}\subseteq B_{\varepsilon}^{G}$. Moreover, if~$\varepsilon_{0}$ is sufficiently small, for all~$\varepsilon<\varepsilon_{0}$ the multiplication map restricted to~$B_{\varepsilon}^{G_{-}}\times B_{\varepsilon}^{P_{0}}$ will be a diffeomorphism onto an open subset of~$G$, and, using compactness of~$\Gamma G_{-}$ once more, for all~$x\in\Gamma G_{-}$, the map
  \begin{equation*}
    B_{\varepsilon}^{G_{-}}\times B_{\varepsilon}^{P_{0}}\to\Gamma\backslash G,\quad(u,b)\mapsto xub
  \end{equation*}
  will be a diffeomorphism onto its image.

  In what follows, we let~$\varphi_{\varepsilon}^{-}$ and~$\varphi_{\varepsilon}^{0}$ be smooth approximate identities on~$G_{-}$ and~$P_{0}$ respectively. Let~$f\in C_{c}^{\infty}(\Gamma\backslash G)$ and~$E_{f}=\int_{\Gamma\backslash G}f\der m_{\lquot{\Gamma}{G}}$. The statement of the proposition will follow from approximating~$\mu_{a}(f-E_{f})$ by a matrix coefficient for~$a\cdot f-E_{f}$ and thereafter application of effective decay of matrix coefficients. Let~$\mathcal{F}\subseteq G_{-}$ be a fundamental domain for~$\Gamma G_{-}$ with compact closure. Then we can find a disjoint, finite collection of subsets~$F_{i}$ of~$\mathcal{F}$, the union of which is conull in~$\mathcal{F}$ and so that for sufficiently small~$\delta>0$ the map~$F_{i}B_{\delta}^{G_{-}}\times B_{\delta}^{P_{0}}\to\Gamma\backslash G$ mapping~$(u,b)$ to~$\Gamma ub$ is injective for each~$i$. Note that the choice of these subsets and~$\delta$ is independent of the element~$a$. 
    
  Let~$\chi_{F_{i}}$ denote the indicator function of~$F_{i}$. Consider the quantity
  \begin{equation*}
    I_{i}=\frac{1}{\mathrm{vol}(F_{i})}\int_{G_{-}}\varphi_{\varepsilon}^{-}\star\chi_{F_{i}}(u)f(\Gamma ua)\der u.
  \end{equation*}
  Then we have
  \begin{equation*}
    \sum_{i=1}^{n}\frac{\mathrm{vol}(F_{i})}{\mathrm{vol}(\mathcal{F})}I_{i}=\frac{1}{\mathrm{vol}(\mathcal{F})}\int_{G_{-}}\varphi_{\varepsilon}^{-}\star\chi_{\mathcal{F}}(u)f(\Gamma ua)\der u=\mu_{a}(f).
  \end{equation*}
  Hence it suffices to prove a bound for the quantity~$I_{i}$ for some fixed~$i$. Using Lemma~\ref{lem:Lipshitzbound} we can find some~$\Ltwo$-Sobolev norm~$\mathcal{S}$ such that
  \begin{equation*}
    \bigg\lvert\int_{P_{0}}\varphi_{\varepsilon}^{0}(b)(f(\Gamma ua)-f(\Gamma uba))\der b\bigg\rvert\leq\mathcal{S}(f)\int_{P_{0}}\varphi_{\varepsilon}^{0}(b)d(1,a^{-1}ba)\der b\leq\varepsilon\mathcal{S}(f),
  \end{equation*}
  as~$b$ is defined by an element in~$\mathfrak{g}$, which is a sum of eigenvectors for~$\log a$ for positive eigenvalues. We can without loss of generality assume that the Sobolev norm was defined using a basis of~$\mathfrak{g}_{-}$ and a basis of~$\mathfrak{p}_{0}$. Let~$\psi_{i}\in C_{c}^{\infty}(\Gamma\backslash G)$ be the function defined on~$\Gamma F_{i}B_{\delta}^{G_{-}}B_{\delta}^{P_{0}}$ by~$\psi_{i}(\Gamma ub)=\frac{1}{\mathrm{vol}(F_{i})}\varphi_{\varepsilon}^{-}\star\chi_{F_{i}}(u)\varphi_{\varepsilon}^{0}(b)$, and~$0$ outside. By the choice of~$\varepsilon_{0}$, this is well-defined and smooth. As the Haar measure on~$G$ decomposes as a product of the Haar measures on $P_{0}$ and $G_{-}$ on a neighborhood of the identity, it follows that
  \begin{align*}
    \lvert I_{i}-E_{f}\rvert&\ll\varepsilon\mathcal{S}(f)+\bigg\lvert\frac{1}{\mathrm{vol}(F_{i})}\int_{P_{0}}\int_{G_{-}}\varphi_{\varepsilon}^{0}(b)\varphi_{\varepsilon}^{-}\star\chi_{F_{i}}(u)f(\Gamma uba)\der u\der b-E_{f}\bigg\rvert\\
    &=\varepsilon\mathcal{S}(f)+\bigg\lvert\int_{\Gamma\backslash G}\psi_{i}(x)f(xa)\der x-\int_{\Gamma\backslash G}f\bigg\rvert\\
    &\ll\varepsilon\mathcal{S}(f)+e^{-\kappa_{\tau}\kappa_{H}\rho_{+}(\log a)}\varepsilon^{-\ell-\frac{1}{2}\dim(G)}\mathcal{S}(f),
  \end{align*}
  where the final bound is obtained combining the bounds for the Sobolev norm of~$\psi_{i}$---here~$\ell$ is the degree of~$\mathcal{S}$---with the decay of matrix coefficients \eqref{eq:matrixcoefficients} and \eqref{eq:boundharishchandra}. Now we can choose~
  \begin{equation*}
    \varepsilon=e^{-\frac{\kappa_{\tau}\kappa_{H}}{\ell+1+\frac{1}{2}\dim(G)}\varrho_{+}(\log a)}
  \end{equation*}
  to obtain the claim with~$\kappa=\frac{\kappa_{\tau}\kappa_{H}}{\ell+1+\frac{1}{2}\dim(G)}$.
\end{proof}
The method of proof, when examined a bit more carefully, yields the following strengthening.
\begin{corollary}\label{cor:jordanmeasurablebanana}
  Let~$a\in A^{+}$ be an element of the Cartan subgroup and assume that~$\Gamma G_{-}\subseteq\lquot{\Gamma}{G}$ is a closed orbit. Let~$F\subseteq G_{-}$ a positive measure subset which is injective for the quotient map~$g\mapsto\Gamma g$ and such that the~$\varepsilon$-neighborhood of~$\partial F$ is bounded by a constant times~$\varepsilon$ for~$\varepsilon$ in a bounded set. Then
  \begin{equation*}
    \bigg\lvert\frac{1}{\vol{\Gamma Fa}}\int_{\Gamma G_{-}a}f\cdot\chi_{\Gamma Fa}-\int_{\lquot{\Gamma}{G}}f\bigg\rvert\ll \frac{1}{\sqrt{\vol{F}}}e^{-\kappa\rho_{+}(\log a)}\Scal(f)
  \end{equation*}
\end{corollary}
\begin{proof}
  Recall that~$G_{-}$ is unipotent and hence the exponential map is a polynomial and the Lebesgue measure on the Lie algebra defines a Haar measure on~$G_{-}$.

  If~$\varphi_{\varepsilon}^{-}$ is a smooth approximate identity for~$G_{-}$, then \eqref{eq:sobolevembedding} and the fact that~$\varphi_{\varepsilon}^{-}\star\chi_{F}$ differs from~$\chi_{F}$ only on a~$\varepsilon$-neighbourhood of~$\partial F$ can be used to obtain
  \begin{equation*}
    \bigg\lvert\frac{1}{\vol{\Gamma Fa}}\int_{\Gamma Fa}f(x)\der\mu_{a}(x)-\frac{1}{\vol{F}}\int_{G_{-}}\varphi_{\varepsilon}^{-}\star\chi_{F}(u)f(\Gamma ua)\der u\bigg\rvert\ll\varepsilon\Scal(f).
  \end{equation*}
  On a neighbourhood of~$F$, every element~$g$ in~$G$ has a unique decomposition as~$g=ub$ with~$u\in G_{-}$ and~$b\in P_{0}$. Now let~$\psi:\lquot{\Gamma}{G}\to\R$ be the function given on~$\Gamma FB_{\delta}^{G_{-}}B_{\delta}^{P_{0}}$ by
  \begin{equation*}
    \psi(\Gamma ub)=\tfrac{1}{\vol{F}}\varphi_{\varepsilon}^{-}\star\chi_{F}(u)\varphi_{\varepsilon}^{0}(b)
  \end{equation*}
  and extended by~$0$ outside. Using a basis respecting the splitting~$\lieg=\lieg_{-}\oplus\liep_{0}$ and applying Jensen's inequality, one obtains that~$\mathcal{S}(\psi)\ll\frac{1}{\vol{F}}\mathcal{S}(\varphi_{\varepsilon})$. From there the proof proceeds as above.
\end{proof}

Combining Proposition \ref{prop:banana} once more with the effective decay of matrix coefficients, we can deduce effective equidistribution of certain discrete subsets of closed horospherical orbits. The example to keep in mind is the set of integer points on the periodic orbit of length~$n$ discussed in Corollary \ref{cor:integerpoints}, which will be examined more closely in~\cite{sl2}. Again it will prove useful to formulate the statement in more general terms. The interpretation in the context of this article is postponed to Section~\ref{sec:rationalpoints}. For what follows, we call a set~$F\subseteq G_{-}$ injective on~$V\subseteq\Gamma\backslash G$, if on~$F\times V$ the map~$(u,x)\mapsto xu$ is injective. We denote by~$C(G_{-})$ the center of~$G_{-}$.

\begin{proposition}\label{prop:formularationalpoints}
  There exist an~$\Ltwo$-Sobolev norm~$\mathcal{S}$ on~$C_{c}^{\infty}(\Gamma\backslash G)$ and~$\eta\in(0,1)$ such that the following is true. Let~$a\in A_{+}$ and~$\gamma\neq\one$ contained in~$C(G_{-})\cap\Gamma$. Assume that~$\gamma$ does not have finite order. Assume that~$\mathcal{P}\subseteq\Gamma G_{-}a$ is a finite,~$\gamma$-invariant subset, then for all~$f\in C_{c}^{\infty}(\Gamma\backslash G)$, and for all sets~$F\subseteq G_{-}$ injective on~$\mathcal{P}$ containing the identity
  \begin{equation*}
    \bigg\lvert\frac{1}{\lvert\mathcal{P}\rvert}\sum_{x\in\mathcal{P}}f(x)-\int_{\Gamma\backslash G}f\bigg\rvert\ll\bigg(\frac{\mathrm{vol}(\Gamma G_{-})}{\lvert\mathcal{P}\rvert\mathrm{vol}(F)}\bigg)^{\frac{1}{2}}e^{(1-\eta)\rho_{+}(\log a)}\mathcal{S}(f)+\mathrm{diam}(F)\mathcal{S}(f).
  \end{equation*}
\end{proposition}
The proof of Proposition~\ref{prop:formularationalpoints} relies on a spectral gap of the action of~$\gamma$ on~$\lquot{\Gamma}{G}$. As~$\Gamma$ is an irreducible arithmetic lattice and as~$\gamma$ is a rational unipotent inside a copy of~$\SL_{2}$ defined over~$\bQ$, this is a corollary of the work by Burger and Sarnak~\cite[Thm.~1.1~(a)]{BurgerSarnak}. For the sake of completeness, we will give an explicit argument.
\begin{lemma}
  Let~$\gamma$ as in Proposition~\ref{prop:formularationalpoints}. Then there is some non-zero~$k\in\N$ such that~$\gamma^{k}\in\G(\bZ)$.
\end{lemma}
\begin{proof}
  By assumption,~$\Gamma$ is an arithmetic lattice and hence~$\rquot{\Gamma}{\Gamma\cap\G(\bZ)}$ is finite. Moreover,~$\gamma$ acts on~$\rquot{\Gamma}{\Gamma\cap\G(\bZ)}$ by multiplication on the left. By the pigeon-hole principle, we can find~$m<n\in\N$ such that
  \begin{equation*}
    \gamma^{m}\big(\Gamma\cap\G(\bZ)\big)=\gamma^{n}\big(\Gamma\cap\G(\bZ)\big)
  \end{equation*}
  and thus letting~$k=n-m$ we get~$\gamma^{k}\in\G(\bZ)$.
\end{proof}
We let~$X=\log\gamma^{k}$. As~$G_{-}$ is unipotent,~$X$ is a polynomial in~$\gamma^{k}$ with rational coefficients and therefore~$X\in\lieg_{\bQ}$. Furthermore we have~$\exp\frac{X}{k}=\gamma$.
\begin{lemma}\label{lem:rationality}
  There exists a~$\bQ$-subgroup~$H\leq G$ and an isogeny~$\SL_{2}(\bR)\to H$ such that~$\gamma$ is contained in the image of the upper triangular unipotent subgroup.
\end{lemma}
\begin{proof}
  Using the Jacobson-Morozov theorem \cite[Ch.~VIII,\textsection11.2,Prop.~2]{Bourbaki} there is an~$\liesltwo$-triple~$(X,E,Y)$ in~$\lieg_{\bQ}$, i.e.~$X,E,Y$ satisfy
  \begin{equation*}
    [E,X]=2X,\quad[E,Y]=-2Y,\quad[X,Y]=-E.
  \end{equation*}
  As~$\gamma\neq\mathbbm{1}$ by assumption, we have that~$X$ is non-zero and thus~$(X,E,Y)$ generates a Lie subalgebra of~$\lieg$ isomorphic to~$\liesltwo$. As of \cite[Cor.~7.9]{Borel}, this subalgebra is the Lie algebra of a~$\bQ$-subgroup of~$\G$. Hence the claim.
\end{proof}
\begin{corollary}\label{cor:decayofmatrixcoefficientsforgamma}
  For all~$\varepsilon\in(0,\frac{1}{2})$, for all~$k\in\Z$, and for all~$f_{1},f_{2}\in C_{c}(\lquot{\Gamma}{G})$ we have
\begin{equation*}
  \bigg\lvert\langle\gamma^{k}f_{1},f_{2}\rangle-\int_{\lquot{\Gamma}{G}}f_{1}\int_{\lquot{\Gamma}{G}}\overline{f_{2}}\bigg\rvert\ll_{\varepsilon}(1+\lvert k\rvert)^{-1+\varepsilon}\mathcal{S}(f_{1})\mathcal{S}(f_{2}),
\end{equation*}
where~$\mathcal{S}$ denotes a degree-$\ell$~$\Ltwo$-Sobolev norm on~$C_{c}^{\infty}(\lquot{\Gamma}{G})$.
\end{corollary}
\begin{proof}
  Using Lemma~\ref{lem:rationality} and Theorem~1.1~(a) in \cite{BurgerSarnak} combined with uniformity of the spectral gap on congruence quotiens of~$\SL_{2}(\bR)$ \cite{Selberg65}, we get by standard techniques for the estimation of matrix coefficients that there is some~$\Ltwo$-Sobolev norm~$\mathcal{S}$ on~$\smooth(\lquot{\Gamma}{G})$ such that
  \begin{equation*}
    \big\lvert \langle\gamma^{k}f_{1},f_{2}\rangle-E_{f_{1}}\overline{E_{f_{2}}}\big\rvert\ll_{\varepsilon}(1+\lvert k\rvert)^{-1+\varepsilon}\mathcal{S}(f_{1})\mathcal{S}(f_{2}).
  \end{equation*}
\end{proof}
\begin{proof}[Proof of Prop.~\ref{prop:formularationalpoints}]
  We will adapt the argument from \cite{Akka2016}. Given a real-valued function~$f\in C_{c}^{\infty}(\Gamma\backslash G)$ and some~$K\in\mathbb{N}$, denote by~$D_{K}f$ the function given by
  \begin{equation}\label{eq:unipotentdiscrepancy}
    D_{K}f(x)=\frac{1}{K}\sum_{k=0}^{K-1}f(x\gamma^{k})-E_{f}\quad(x\in\Gamma\backslash G).
  \end{equation}
  Let~$F\subseteq G_{-}$ be a positive measure subset containing the identity, then using  and Lemma~\ref{lem:Lipshitzbound} there is some~$\Ltwo$-Sobolev norm~$\mathcal{S}$ such that for all~$x\in\mathcal{P}$
  \begin{equation*}
    \bigg\lvert f(x)-\frac{1}{\mathrm{vol}(F)}\int_{F}f(xu)\der u\bigg\rvert\ll\mathrm{diam}(F)\mathcal{S}(f)
  \end{equation*}
  for an~$\Ltwo$-Sobolev norm of sufficiently large degree. As~$\gamma$ is central in~$G_{-}$ and because~$\mathcal{P}$ is~$\gamma$-invariant, we have for all~$F\subseteq G_{-}$ and for all~$k\in\mathbb{N}$ that
  \begin{equation*}
    \sum_{x\in\mathcal{P}}\int_{F}f(xu)\der u=\sum_{x\in\mathcal{P}}\int_{F}f(xu\gamma^{k})\der u.
  \end{equation*}
  Hence by construction of~$D_{K}$ and for~$F$ of diameter~$\varepsilon_{F}$ and injective on~$\mathcal{P}$, using Cauchy-Schwarz it follows that
  \begin{align}
    \label{eq:boundbyunipotentaverage}\bigg\lvert\frac{1}{\lvert\mathcal{P}\rvert}\sum_{x\in\mathcal{P}}f(x)-E_{f}\bigg\lvert&\ll\varepsilon_{F}\mathcal{S}(f)+\bigg\lvert\frac{1}{\lvert\mathcal{P}\rvert\mathrm{vol}(F)}\sum_{x\in\mathcal{P}}\int_{F}D_{K}f(xu)\der u\bigg\rvert\\
    \notag&\ll\varepsilon_{F}\mathcal{S}(f)+\bigg(\frac{\mathrm{vol}(\Gamma G_{-}a)}{\lvert\mathcal{P}\rvert\mathrm{vol}(F)}\bigg)^{\frac{1}{2}}\mu_{a}\big(\lvert D_{K}f\rvert^{2}\big)^{\frac{1}{2}}.
  \end{align}
  As~$f$ is real-valued, smooth, and of compact support, so is~$(D_{K}f)^{2}-E_{f}^{2}$. In particular, we can apply Proposition \ref{prop:banana} to obtain
  \begin{align}
    \label{eq:boundbyspaceaverage}\bigg\lvert\mu_{a}\big((D_{K}f)^{2}\big)-\int_{\Gamma\backslash G}( D_{K}f)^{2}\bigg\rvert&\ll e^{-\kappa\rho_{+}(\log a)}\mathcal{S}\big((D_{K}f)^{2}-E_{f}^{2}\big)\\
    \notag&\ll e^{-\kappa\rho_{+}(\log a)}\mathcal{S}_{2}(D_{K}f)^{2},
  \end{align}
  where~$\mathcal{S}_{2}$ is an~$\Ltwo$-Sobolev norm satisfying~$\mathcal{S}\ll\mathcal{S}_{2}$. Using the properties of~$\Ltwo$-Sobolev norms described in Section \ref{sec:sobolevnorms}, we find that
  \begin{equation*}
    \mathcal{S}_{2}(D_{K}f)\leq\frac{1}{K}\sum_{k=0}^{K-1}\mathcal{S}_{2}(\gamma^{k}\cdot f)\ll\frac{1}{K}\sum_{k=0}^{K-1}\lVert\gamma^{k}\rVert^{\iota}\mathcal{S}_{2}(f),
  \end{equation*}
  where~$\iota\geq 1$ is a positive constant depending on the degree of~$\mathcal{S}_{2}$. As~$\gamma$ is unipotent, and as~$\pi$ is a rational representation, the same holds for~$\pi(\gamma)$, and hence we obtain~$\lVert\gamma^{k}\rVert\ll(1+\lvert k\rvert)^{N}$ with an implicit constant depending on~$\gamma$ and~$\pi$. Note that~$N\geq 1$. Plugging this into the above bound, we obtain
  \begin{equation}\label{eq:boundsobolevnormdiscrepancy}
    \mathcal{S}_{2}(D_{K}f)\ll\frac{1}{K}\sum_{k=0}^{K-1}(1+k)^{N\iota}\mathcal{S}_{2}(f)\ll(1+K)^{N\iota}\mathcal{S}_{2}(f).
  \end{equation}
  Therefore we have shown
  \begin{align}
    \label{eq:intermediatebound}
      \bigg\lvert\frac{1}{\lvert\mathcal{P}\rvert}\sum_{x\in\mathcal{P}}f(x)-E_{f}\bigg\rvert&\ll\bigg\{\varepsilon_{F}+\bigg(\frac{\mathrm{vol}(\Gamma G_{-}a)}{\lvert\mathcal{P}\rvert\mathrm{vol}(F)}\bigg)^{\frac{1}{2}}e^{-\frac{\kappa}{2}\rho_{+}(\log a)}K^{N\iota}\bigg\}\mathcal{S}_{2}(f)\\
  \notag&\qquad+\bigg(\frac{\mathrm{vol}(\Gamma G_{-}a)}{\lvert\mathcal{P}\rvert\mathrm{vol}(F)}\bigg)^{\frac{1}{2}}\bigg(\int_{\lquot{\Gamma}{G}}(D_{K}f)^{2}\mathrm{d}m_{\lquot{\Gamma}{G}}\bigg)^{\frac{1}{2}}
  \end{align}
  
  Combining Corollary~\ref{cor:decayofmatrixcoefficientsforgamma} and Proposition~\ref{prop:effectivevonneumann}, we have
  \begin{equation}\label{eq:discrepancybound}
    \int_{\Gamma\backslash G}\big( D_{K}f(x)\big)^{2}\der x\ll K^{-\varsigma}\mathcal{S}_{3}(f)^{2}
  \end{equation}
  for some~$\varsigma>0$ and an~$\Ltwo$-Sobolev norm~$\mathcal{S}_{2}\ll\mathcal{S}_{3}$ of sufficiently large degree~$\ell$. Note that the implicit constant depends on~$\gamma$ and on~$\varsigma$. 
  Finally, combining the bounds \eqref{eq:boundbyunipotentaverage}, \eqref{eq:boundbyspaceaverage}, \eqref{eq:boundsobolevnormdiscrepancy}, and \eqref{eq:discrepancybound}, we are left with
  \begin{equation*}
    \bigg\lvert\frac{1}{\lvert\mathcal{P}\rvert}\sum_{x\in\mathcal{P}}f(x)-E_{f}\bigg\rvert\ll\bigg\{\varepsilon_{F}+\bigg(\frac{\mathrm{vol}(\Gamma G_{-}a)}{\lvert\mathcal{P}\rvert\mathrm{vol}(F)}\bigg)^{\frac{1}{2}}\psi(K)\bigg\}\mathcal{S}_{3}(f),
  \end{equation*}
  where~$\psi$ denotes the function
  \begin{equation}\label{eq:equating}
    \psi(K)=e^{-\frac{\kappa}{2}\rho_{+}(\log a)}K^{N\iota}+K^{-\frac{\varsigma}{2}}.
  \end{equation}
  Equating the terms, we choose~$K\asymp e^{\frac{\kappa}{\varsigma+2N\iota}\varrho_{+}(\log a)}$ to obtain~$\psi(K)\ll e^{-\eta\rho_{+}(\log a)}$, where
  \begin{equation}\label{eq:choiceeta}
    \eta=\tfrac{\kappa}{2}\tfrac{\varsigma}{2N\iota+\varsigma}>0.
  \end{equation}
  Note that~$\mathrm{vol}(\Gamma G_{-}a)=e^{2\rho_{+}(\log a)}\mathrm{vol}(\mathcal{F})$, because~$a^{-1}\mathcal{F}a$ is a fundamental domain for the orbit~$\Gamma G_{-}a$. Thus follows
  \begin{equation*}
    \frac{\mathrm{vol}(\Gamma G_{-}a)}{\lvert\mathcal{P}\rvert\mathrm{vol}(F)}=\frac{\mathrm{vol}(\mathcal{F})}{\lvert\mathcal{P}\rvert\mathrm{vol}(F)}e^{2\rho_{+}(\log a)}.
  \end{equation*}
  Collecting terms, this shows
  \begin{equation*}
    \bigg\lvert\frac{1}{\lvert\mathcal{P}\rvert}\sum_{x\in\mathcal{P}}f(x)-E_{f}\bigg\rvert\ll\bigg(\frac{\mathrm{vol}(\mathcal{F})}{\lvert\mathcal{P}\rvert\mathrm{vol}(F)}\bigg)^{\frac{1}{2}}e^{(1-\eta)\rho_{+}(\log a)}\mathcal{S}_{2}(f)+\mathrm{diam}(F)\mathcal{S}_{2}(f)
  \end{equation*}
  as desired.
\end{proof}
For the sake of illustration, we apply the formula to~$\mathrm{SL}_{2}(\mathbb{R})$ and the lattice~$\Gamma=\mathrm{SL}_{2}(\mathbb{Z})$. Then~$\mathfrak{sl}_{2}(\mathbb{R})$ is the space of traceless two-by-two matrices, and is given by the span of the triple
\begin{equation}\label{eq:sl2triple}
  H=\begin{pmatrix}-1 & 0 \\ \phantom{-}0 & 1\end{pmatrix},\quad X=\bigg(\begin{matrix}0 & 1 \\ 0 & 0\end{matrix}\bigg),\quad Y=\bigg(\begin{matrix}0 & 0 \\ 1 & 0\end{matrix}\bigg).
\end{equation}
The subspace~$\mathfrak{a}=\mathbb{R}H$ is a Cartan subalgebra. Of course~$(X,H,Y)$ is an~$\liesltwo$-triple, i.e.~one calculates~$[H,X]=-2X$ and~$[H,Y]=2Y$. Let~$\mathfrak{a}_{+}=\mathbb{R}_{>0}H$, so that~$\Sigma_{+}$ is the singleton containing only the root defined by~$\lambda(H)=2$. It follows that
\begin{equation*}
  G_{-}=\bigg\{\bigg(\begin{matrix}1 & t \\ 0 & 1\end{matrix}\bigg);t\in\mathbb{R}\bigg\},\quad P_{0}=\bigg\{\bigg(\begin{matrix}\alpha & 0 \\ b & \alpha^{-1}\end{matrix}\bigg);\alpha>0,b\in\mathbb{R}\bigg\}
\end{equation*}
Given~$n\in\mathbb{N}$, consider the matrices
\begin{equation*}
  a_{n}=\exp\big((\tfrac{1}{2}\log n)H\big)\in A_{+}\text{ and }\gamma=\bigg(\begin{matrix}1 & 1 \\ 0 & 1\end{matrix}\bigg)\in C(G_{-})\cap\Gamma.
\end{equation*}
Then~$a_{n}\gamma^{k}a_{n}^{-1}=\big(\begin{smallmatrix}1 & k/n \\ 0 & 1\end{smallmatrix}\big)$. Let~$\mathcal{P}=\{\Gamma a_{n}\gamma^{k};0\leq k<n\}$ and~$F=(-\varepsilon,\varepsilon)$. Then~$\mathcal{P}$ is~$\gamma$-invariant and~$\lvert\mathcal{P}\rvert=n$. Proposition \ref{prop:formularationalpoints} yields for smooth~$f\in\mathscr{C}_{c}^{\infty}(\Gamma\backslash G)$
\begin{equation*}
  \bigg\lvert\frac{1}{n}\sum_{k=1}^{n}f(\Gamma a\gamma^{k})-\int_{\Gamma\backslash G}f\bigg\rvert\ll\varepsilon^{-\frac{1}{2}}n^{-\frac{\eta}{2}}\mathcal{S}(f)+\varepsilon\mathcal{S}(f)
\end{equation*}
and using~$\varepsilon=n^{-\frac{\eta}{2}}$ we obtain the following effective equidistribution result:
\begin{corollary}\label{cor:integerpoints}
  There exist an~$\Ltwo$-Sobolev norm~$\mathcal{S}$ on~$\mathrm{SL}_{2}(\mathbb{Z})\backslash\mathrm{SL}_{2}(\mathbb{R})$ and~$\beta>0$ such that denoting
  \begin{equation*}
    a_{y}=\bigg(\begin{matrix}y^{-\frac{1}{2}} & 0 \\ 0 & y^{\frac{1}{2}}\end{matrix}\bigg)\text{ and }u_{t}=\bigg(\begin{matrix}1 & t \\ 0 & 1\end{matrix}\bigg)\quad(t\in\mathbb{R},y>0)
  \end{equation*}
  and for all~$n\in\mathbb{N}$ and~$f\in C_{c}^{\infty}(\mathrm{SL}_{2}(\mathbb{Z})\backslash\mathrm{SL}_{2}(\mathbb{R}))$ holds
  \begin{equation*}
    \bigg\lvert\frac{1}{n}\sum_{k=0}^{n-1}f(\SL_{2}(\mathbb{Z})u_{k/n}a_{n})-\int_{\mathrm{SL}_{2}(\mathbb{Z})\backslash\mathrm{SL}_{2}(\mathbb{R})}f\bigg\rvert\ll n^{-\beta}\mathcal{S}(f)
  \end{equation*}
\end{corollary}
\section[Dirichlet's unit theorem and horospheres]{The ring of integers, Dirichlet's unit theorem and horospheres in Hilbert modular surfaces}\label{sec:ringofintegers}
In this section we will resume the notation from the introduction, i.e.
\begin{itemize}
\item~$\field$ is a totally real number field of degree~$d$ over~$\mathbb{Q}$,
\item~$\mathfrak{o}$ is the ring of integers in~$\field$,
\item~$\mathfrak{o}^{\times}$ is the group of units in~$\mathfrak{o}$,
\item~$\mathfrak{o}_{>0}^{\times}$ is the group of totally positive units in~$\mathfrak{o}$,
\item~$G=\mathrm{SL}_{2}(\mathbb{R})^{d}$
\item and~$\Gamma\leq G$ is the image of~$\mathrm{SL}_{2}(\mathfrak{o})$ obtained by diagonal embedding with respect to the distinct Galois embeddings~$\field\hookrightarrow\mathbb{R}$.
\end{itemize}
\subsection{Algebraic Properties of~$\mathfrak{o}$ and its Quotient Rings}\label{sec:ant}
In this section we discuss properties of the ring of integers in a totally real number field. In particular we describe some asymptotic properties of the Euler totient function.
\begin{definition}
  Let~$I\subseteq\mathfrak{o}$ be an ideal. Then~$\phi(I)$ is defined to be the number
  \begin{equation*}
    \phi(I)=\lvert(\mathfrak{o}/I)^{\times}\rvert,
  \end{equation*}
  i.e.~the number of units in the ring~$\mathfrak{o}/I$. Given~$y\in\mathfrak{o}$, we let~$\phi(y)=\phi(y\mathfrak{o})$.
\end{definition}
We denote by~$N(I)$ the index of~$I$ in~$\mathfrak{o}$ and point out that this is finite. Indeed~$I\subseteq\mathbb{R}^{d}$ is a lattice and~$\mathrm{covol}(I)=N(I)\mathrm{covol}(\mathfrak{o})$.
\begin{remark}
  Any non-trivial ideal~$I\subseteq\mathfrak{o}$ has a unique prime factorization
  \begin{equation*}
    I=\prod_{\liep\text{ prime}}\liep^{\nu_{\liep}(I)}
  \end{equation*}
  with~$\nu_{\liep}(I)\in\N_{0}$ equal to~$0$ for almost every~$\liep\subseteq\mathfrak{o}$ prime. Using the Chinese remainder theorem, it is thus sufficient to determine~$\phi(\liep^{n})$ for~$n\in\N$ and~$\liep\subseteq\mathfrak{o}$ prime.
\end{remark}
Let~$\mathfrak{o}_{\liep}$ be the localization of~$\mathfrak{o}$ at~$\liep$, then~$\mathfrak{o}_{\liep}$ is a local ring and its unique unique maximal ideal~$\mathfrak{m}_{\liep}=\liep\mathfrak{o}_{\liep}$. One can show that for all~$m\geq 1$
  \begin{equation*}
    \rquot{\mathfrak{o}}{\liep^{m}}\cong\rquot{\mathfrak{o}_{\liep}}{\mathfrak{m}_{\liep}^{m}}.
  \end{equation*}
  On the other hand, if~$R$ is a local Dedekind domain with unique maximal ideal~$\liem$, if the cardinality~$q=\lvert\rquot{R}{\liem}\rvert$ of the quotient is finite, and if~$\liem^{n}\subsetneq\liem^{n-1}$, then
  \begin{equation*}
    \big\lvert\left(\rquot{R}{\liem^{n}}\right)^{\times}\big\rvert=(q-1)q^{n-1}.
  \end{equation*}
  As a corollary, it follows that~$\phi(\liep^{n})=(N(\liep)-1)N(\liep)^{n-1}$ for~$\liep$ prime in~$\mathfrak{o}$.
\begin{corollary}
  Let~$I\subseteq\mathfrak{o}$ an ideal, then
  \begin{equation*}
    \phi(I)=\prod_{\nu_{\liep}(I)\neq0}\big(N(\liep)-1\big)N(\liep)^{\nu_{\liep}(y)-1}
  \end{equation*}
\end{corollary}
Similarly to the situation for~$\Q$, one can prove the following
\begin{proposition}
  Let~$\varepsilon>0$, then for all ideals~$I$ with~$N(I)$ sufficiently large
  \begin{equation*}
    N(I)^{1-\varepsilon}<\phi(I)<N(I).
  \end{equation*}
\end{proposition}
\begin{proof}
  Note that for all~$M>0$ there are only finitely many ideals~$I\subseteq\mathfrak{o}$ of index bounded by~$M$. We provide a geometric argument for this. Fix a norm~$\lVert\cdot\rVert$ on~$\R^{d}$ and let~$\delta=\min\{\lVert v\rVert;v\in\mathfrak{o}\}$. Let~$I\subseteq\mathfrak{o}$ be an ideal, then by Minkowski's theorem on successive minima~$I$ has a basis~$v_{1},\ldots,v_{d}$ such that~$\lVert v_{i}\rVert\asymp \lambda_{i}(I)$, where
  \begin{equation*}
    \lambda_{i}(I)=\min\set{r>0}{I\cap \overline{B_{r}(0)}\text{ contains~$i$ linearly independent vectors}}
  \end{equation*}
  and~$\lambda_{1}(I)\cdots\lambda_{d}(I)\asymp\covol{I}$. Hence there is~$C>0$ such that
  \begin{equation*}
    \lambda_{d}(I)\leq C\delta^{1-d}\mathrm{covol}(I)=C\delta^{1-d}\covol{\mathfrak{o}}N(I).
  \end{equation*}
  For every~$M>0$ there are only finitely many vectors~$v\in\mathfrak{o}$ such that
  \begin{equation*}
    \lVert v\rVert\leq C\delta^{1-d}\covol{\mathfrak{o}}M.
  \end{equation*}
  This proves the claim in the beginning. In particular, for every~$M>0$, there are only finitely many prime ideals of norm at most~$M$. Using this, one can prove the proposition in exactly the same way as for the totient function on~$\mathbb{Z}$. We refer the reader to~\cite[\textsection18.4]{HardyWright}.
\end{proof}
Finally, we can deduce the following
\begin{proposition}\label{prop:totientbound}
  There exists~$M>0$ such that for all~$y\in\mathfrak{o}$ holds
  \begin{equation*}
    \frac{\lvert N(y)\rvert}{\phi(y)(\log\log\lvert N(y)\rvert)^{d}}\leq M.
  \end{equation*}
\end{proposition}
\begin{proof}
  Recall that~$\lvert N(y)\rvert=N(y\mathfrak{o})$. Using the preceding discussion and multiplicativity of the norm, we obtain that
  \begin{equation*}
    \frac{\phi(y)}{\lvert N(y)\rvert}=\prod_{\nu_{\liep}(y)\neq0}\frac{N(\liep)-1}{N(\liep)}\geq\prod_{\substack{\liep\subseteq\mathfrak{o}\text{ prime}\\ N(\liep)\leq \lvert N(y)\rvert}}\frac{N(\liep)-1}{N(\liep)}.
  \end{equation*}
  Any prime ideal in~$\mathfrak{o}$ has norm equal to~$p^{k}$ for some prime~$p\in\Z$ and some~$k\in\N$, because any prime ideal in~$\mathfrak{o}$ is maximal. In fact, the prime~$p$ occuring is exactly the prime the prime ideal is lying over in the sense that its intersection with the rational integers yield the ideal~$(p)$ in~$\Z$. On the other hand,~$p\mathfrak{o}$ has norm~$p^{d}$, and thus using the fact that the ideal~$p\mathfrak{o}$ is a product of the prime ideals lying over~$p$ together with the multiplicativity of the norm, it follows that the norm of a prime ideal lying over~$p$ is at most~$p^{d}$. The fundamental identity relating the inertia degree and the ramification index of prime ideals lying over~$p$ to the degree of the number field \cite[Ch.~I,~\textsection8]{Neukirch} implies that there are at most~$d$ prime ideals in~$\mathfrak{o}$ lying over~$p$. This implies that
  \begin{align*}
    \frac{\phi(y)}{\lvert N(y)\rvert}\geq&\prod_{\substack{\liep\subseteq\mathfrak{o}\text{ prime}\\ N(\liep)\leq \lvert N(y)\rvert}}\frac{N(\liep)-1}{N(\liep)}\geq\prod_{\substack{p\in\N\text{ prime}\\ p\leq\lvert N(y)\rvert}}\prod_{\substack{\liep\subseteq\mathfrak{o}\text{ prime}\\ \liep\text{ lies over }p}}\frac{N(\liep)-1}{N(\liep)}\\
    \geq&\prod_{\substack{p\in\N\text{ prime}\\ p\leq\lvert N(y)\rvert}}\prod_{k=1}^{d}\left(\frac{p^{k}-1}{p^{k}}\right)^{d}\geq \zeta(2)^{-d}\cdots\zeta(d)^{-d}\prod_{\substack{p\in\N\text{ prime}\\ p\leq\lvert N(y)\rvert}}\left(\frac{p-1}{p}\right)^{d}.
  \end{align*}
  So we are left with bounding the product on the right-hand side. Here we proceed as in the proof for the corresponding~$\lim\inf$ for the Euler totient function on~$\Z$ \cite[Proof of Thms.~323 and 328]{HardyWright}, which yields that
  \begin{equation*}
    \prod_{\substack{p\in\N\text{ prime}\\ p\leq\lvert N(y)\rvert}}\left(\frac{p-1}{p}\right)\geq\left(1-\frac{1}{\log\lvert N(y)\rvert}\right)^{\frac{\log\lvert N(y)\rvert}{\log\log\lvert N(y)\rvert}}\prod_{\substack{p\in\N\text{ prime}\\ p\leq\log\lvert N(y)\rvert}}\left(\frac{p-1}{p}\right).
  \end{equation*}
  Using Mertens' theorem (cf.~\cite[Thm.~429]{HardyWright}), we know that for large enough~$\lvert N(y)\rvert$ holds
  \begin{equation*}
    \prod_{\substack{p\in\N\text{ prime}\\ p\leq\log\lvert N(y)\rvert}}\left(\frac{p-1}{p}\right)\asymp\frac{1}{\log\log\lvert N(y)\rvert},
  \end{equation*}
  whereas the function~$t\mapsto(1-t^{-1})^{t/\log t}$ satisfies
  \begin{equation*}
    \lvert(1-t^{-1})^{t/\log t}-1\rvert\ll(\log t)^{-1}
  \end{equation*}
  as~$t\to\infty$, so that all in all we obtain for large enough~$\lvert N(y)\rvert$, that
  \begin{equation*}
    \prod_{\substack{p\in\N\text{ prime}\\ p\leq\lvert N(y)\rvert}}\left(\frac{p-1}{p}\right)\gg\frac{1}{\log\log\lvert N(y)\rvert}.
  \end{equation*}
  This proves
  \begin{equation*}
    \frac{\lvert N(y)\rvert}{\phi(y)\big(\log\log\lvert N(y)\rvert\big)^{d}}\ll 1
  \end{equation*}
  and hence the claim.
\end{proof}
\subsection{Dirichlet's unit theorem, the Cartan subgroup and effective equidistribution of large horospheres}\label{sec:dirichlet}
We note first, that~$\mathfrak{o}_{>0}^{\times}\leq\mathfrak{o}^{\times}$ is a finite index subgroup as by Dirichlet's unit theorem~$\mathfrak{o}^{\times}/\{\pm1\}$ is torsion free and hence the image of~$\mathfrak{o}^{\times}/\{\pm1\}$ under the map~$x\mapsto x^{2}$ is a~$\Z$-submodule of full rank contained in~$\mathfrak{o}_{>0}^{\times}$.
In fact by noting that under the isomorphism~$\mathfrak{o}^{\times}/\{\pm1\}\cong\Z^{d-1}$ we have
\begin{equation*}
  (2\Z)^{d-1}\leq\mathfrak{o}_{>0}^{\times}\cong\Z^{d-1},
\end{equation*}
we can bound~$[\mathfrak{o}^{\times}:\mathfrak{o}_{>0}^{\times}]\leq 2^{d}$.

In the remainder of this section, we describe a Cartan subgroup of~$G$ explicitly, choose a particularly useful basis for its Lie algebra and lay grounds for exploiting the fact that the~$\Q$-rank of~$G=\SL_{2}(\R)^{d}$ equals~$1$.

The Lie algebra of~$G$ is the direct sum of~$d$ copies of~$\mathfrak{sl}_{2}(\mathbb{R})$, and thus the direct sum of a Cartan subalgebra in each component yields a Cartan subalgebra in~$\mathfrak{g}$. We choose the Cartan subalgebra explicitly as
\begin{equation*}
  \mathfrak{a}=\bigoplus_{i=1}^{d}\mathbb{R}H,
\end{equation*}
where~$H$ is the diagonal element in the standard~$\mathfrak{sl}_{2}$-triple in \eqref{eq:sl2triple}. The exponential map~$\exp:\liea\to G$ defines an isomorphism onto its image~$A$, which is the group of diagonal matrices in~$G$ with positive entries on the diagonal. In what follows, we fix a set of generators~$\varepsilon_{1},\ldots,\varepsilon_{d-1}$ of~$\mathfrak{o}_{>0}^{\times}$. Given~$\varepsilon\in\mathfrak{o}_{>0}^{\times}$, let~$a_{\varepsilon}=\big(\begin{smallmatrix}\varepsilon^{-1} & 0 \\ 0 & \varepsilon\end{smallmatrix}\big)\in\mathrm{SL}_{2}(\mathfrak{o})$. Given an index~$i=1,\ldots,d-1$, we denote by~$h_{i}$ the embedding of~$a_{\varepsilon_{i}}$ in~$G$ and define~$H_{i}\in\mathfrak{a}$ as the logarithm of~$h_{i}$. Finally, we denote by~$H_{d}$ the diagonal embedding of~$H$ in~$\mathfrak{a}$.
\begin{proposition}
  The elements~$H_{1},\ldots,H_{d}$ form a basis of~$\liea$.
\end{proposition}
\begin{proof}
  Recall that~$\liea\cong\R^{d}$ and consider the homomorphism~$\field^{\times}\to\R^{d}$ given by
  \begin{equation*}
    x\mapsto(\log\lvert\sigma_{1}x\rvert,\ldots,\log\lvert\sigma_{d}x\rvert)
  \end{equation*}
  As of Dirichlet's unit theorem (cf. \cite[pp.~41f.]{Neukirch}), we know that the image~$\Jcal$ of~$\mathfrak{o}^{\times}$ is a lattice in the subspace~$E=(1,\ldots,1)^{\perp}\subseteq\mathbb{R}^{d}$. It follows that the image of any set of generators of~$\mathfrak{o}_{>0}^{\times}$ is a basis of the subspace~$E$ and hence the union of this image together with the vector~$(1,\ldots,1)\in\R^{d}$ is a basis of~$\R^{d}$. The map sending~$(v_{1},\ldots,v_{d})$ to the element in~$\mathfrak{g}$ whose coordinates are~$v_{i}H$ yields an isomorphism~$\R^{d}\cong\liea$ and it maps the image of the generator~$\varepsilon_{i}$ to~$H_{i}$.
\end{proof}
We fix the set of positive roots to be the~$d$ distinct roots~$\lambda_{1},\ldots,\lambda_{d}$ sending an element in~$\mathfrak{a}$ to the negative of the top-left entry of one of its components, i.e.~
\begin{equation*}
  \lambda_{i}\Big(\sum_{j=1}^{d}\alpha_{j}H_{j}\Big)=2\sum_{j=1}^{d-1}\alpha_{j}\log\sigma_{i}\varepsilon_{j}+2\alpha_{d}.
\end{equation*}
Each of the roots has multiplicity one, and thus in the notation of Section \ref{sec:horospheres}, we get
\begin{equation}\label{eq:choiceofroots}
  \varrho_{+}\Big(\sum_{j=1}^{d}\alpha_{j}H_{j}\Big)=\sum_{j=1}^{d-1}\alpha_{j}\log\sigma_{i}\varepsilon_{j}+\alpha_{d}.
\end{equation}
We will use the following corollaries to Dirichlet's unit theorem, which allow us to reduce the problem to a rank one problem, i.e.~\emph{$G$ has~$\Q$-rank one}. The first corollary states that element in~$\mathfrak{o}$ is associated to an element that acts as an isothety.
\begin{corollary}\label{cor:rankone}
  Let~$y\in\mathfrak{o}$ be totally positive, then there is~$\varepsilon\in\mathfrak{o}_{>0}^{\times}$ such that
  \begin{equation*}
    \sigma_{i}(\varepsilon y)\asymp N(\varepsilon y)^{\frac{1}{d}}\quad(i=1,\ldots,d).
  \end{equation*}
\end{corollary}
\begin{proof}
  It suffices to prove the existence of~$\varepsilon\in\mathfrak{o}_{>0}^{\times}$ such that for all~$i=1,\ldots,d$ we have 
  \begin{equation*}
    (\sigma_{i} y)^{1-d}\prod_{\sigma\neq\sigma_{i}}\sigma y\asymp\sigma_{i} \varepsilon^{d}
  \end{equation*}
  Let~$z=N(y)y^{-1}$, then~$N(z)=N(y)^{d-1}$ and thus~$y^{1-d}z$ is contained in the norm 1 surface in~$\R^{d}$. Taking logarithms of the required inequality, we need to find~$c,C\in\mathbb{R}$ and~$\varepsilon\in\mathfrak{o}_{>0}^{\times}$ such that
  \begin{equation*}
    c+d\log(\sigma_{i} \varepsilon)\leq \log(\sigma_{i} y^{1-d}z)\leq C+d\log(\sigma_{i} \varepsilon)\quad (1\leq i\leq d).
  \end{equation*}
  The vector with entries~$\log\sigma_{i}y^{1-d}z$ lies in the subspace~$E=(1,\ldots,1)^{\perp}$ and as the constants are allowed to depend on the covolume of~$\mathfrak{o}_{>0}^{\times}$, the statement follows from Dirichlet's unit theorem and the finite index of~$\mathfrak{o}_{>0}^{\times}$ in~$\mathfrak{o}^{\times}$.
\end{proof}
The second corollary states that if we replace an element in~$\mathfrak{o}$ by an associated element, the error introduced can be absorbed in the lattice~$\Gamma$, up to a term bounded independently of our original element.
\begin{corollary}\label{cor:unitdecomposition}
  There is a compact subset~$C\subset A$ such that for all~$\varepsilon\in\mathfrak{o}_{>0}^{\times}$ there are~$\gamma\in\Gamma\cap A$ and~$g\in C$such that~$a_{\alpha}(\varepsilon)=\gamma g$.
\end{corollary}
\begin{proof}
  Using Dirichlet's unit theorem and the discussion of totally positive units, we know that~$\varepsilon=\varepsilon_{1}^{n_{1}}\cdots\varepsilon_{d-1}^{n_{d-1}}$, where the~$\varepsilon_{1},\ldots,\varepsilon_{d-1}\in\mathfrak{o}_{>0}^{\times}$ form a set of generators. It follows that
  \begin{equation*}
    a_{\alpha}(\varepsilon)=\exp_{G}(\underbrace{\alpha n_{1}H_{1}+\cdots+\alpha n_{d-1}H_{d-1}}_{\in H_{d}^{\perp}})=\exp_{G}(v+\lambda)
  \end{equation*}
  where~$\lambda$ lies in the~$\mathbb{Z}$-module~$\Lambda$ generated by~$H_{1},\ldots,H_{d-1}$ and~$v\in E$ is contained in any fundamental domain~$F$ for~$\Lambda$. As~$\Lambda$ is a cocompact lattice in~$\Rspan\{H_{1},\ldots,H_{d-1}\}\cong\R^{d-1}$, we can assume that~$\cl{F}\subseteq\mathfrak{a}$ is compact. 
  Therefore~$a_{\alpha}(\varepsilon)\in\exp_{G}(F)a(z)$ for some~$z\in\mathfrak{o}_{>0}^{\times}$ and as~$\exp_{G}(\cl{F})$ is compact, the claim follows.
\end{proof}
Using these corollaries to Dirichlet's unit theorem, we can prove the following analog of equidistribution of large horospheres:
\begin{proposition}\label{prop:equidistributionoflonghorospheres}
  There exist~$\kappa>0$ and an~$\Ltwo$-Sobolev norm~$\mathcal{S}$ on~$C_{c}^{\infty}(\Gamma\backslash G)$ such that for all~$f\in C_{c}^{\infty}(\Gamma\backslash G)$ and for all~$y\in\mathfrak{o}$ totally positive
  \begin{equation*}
    \bigg\lvert\int_{\Gamma Ua_{\alpha}(y)}f-\int_{\Gamma\backslash G}f\bigg\rvert\ll N(y)^{-\kappa\alpha}S(f)
  \end{equation*}
\end{proposition}
\begin{proof}
  Let~$\varepsilon$ be any totally positive unit. As of Corollary \ref{cor:unitdecomposition}, we have~$a_{\alpha}(\varepsilon)=\gamma g$ for~$\gamma\in\Gamma\cap A,g\in C\cap A$, where~$C$ is a fixed, compact subset of~$G$. As~$A$ normalizes~$U$, it follows that~$\Gamma Ua_{\alpha}(\varepsilon y)=\Gamma Ua_{\alpha}(y)g$ and hence
  \begin{equation*}
    \bigg\lvert\int_{\Gamma a_{\alpha}(y)}f-\int_{\Gamma\backslash G}f\bigg\rvert=\bigg\lvert\int_{\Gamma a_{\alpha}(\varepsilon y)}(g\cdot f)-\int_{\Gamma\backslash G}(g\cdot f)\bigg\rvert.
  \end{equation*}
  Using Lemma \ref{lem:shiftedsobolevnorms} and~$N(\varepsilon)=1$ for all~$\varepsilon\in\mathfrak{o}_{>0}^{\times}$, it suffices to prove the statement for some element associated to~$y$ by a totally positive unit. As of Corollary \ref{cor:rankone}, there is~$\delta\in(0,1)$ depending on~$\field$ and~$\alpha$ and a totally positive unit~$\varepsilon\in\mathfrak{o}_{>0}^{\times}$ such that 
\begin{equation*}
  (1-\delta)N(\varepsilon y)^{\frac{\alpha}{d}}\leq\sigma_{i}(\varepsilon y)^{\alpha}\leq(1+\delta)N(\varepsilon y)^{\frac{\alpha}{d}},
\end{equation*}
i.e.~$a_{\alpha}(\varepsilon y)$ is almost a homothety. It follows that there is a compact subset~$M_{\alpha}$ of~$A$ independent of~$y$ and~$\varepsilon$ such that denoting
\begin{equation*}
  a_{\ast}=\exp\big((\tfrac{\alpha}{d}\log N(y))H_{d}\big)
\end{equation*}
we have~$a_{\alpha}(\varepsilon y)=a_{\ast}h$ for some~$h\in M_{\alpha}$. Using Proposition \ref{prop:banana}, Equation \eqref{eq:choiceofroots}, and Lemma \ref{lem:shiftedsobolevnorms}, it follows that for all~$f\in C_{c}^{\infty}(\Gamma\backslash G)$
\begin{align*}
  \bigg\lvert\int_{\Gamma Ua_{\alpha}(\varepsilon y)}f-&\int_{\Gamma\backslash G}f\bigg\rvert=\bigg\lvert\int_{\Gamma Ua_{\ast}}(h\cdot  f)-\int_{\Gamma\backslash G}(h\cdot f)\bigg\rvert\\
  &\ll N(\varepsilon y)^{-\kappa\alpha}\mathcal{S}(h\cdot f)\ll N(y)^{-\kappa\alpha}\mathcal{S}(f).
\end{align*}
\end{proof}
\section{Rational Points}\label{sec:rationalpoints}
In this section we use the formula obtained in Section~\ref{sec:horospheres} to show equidistribution of rational points of a fixed denominator on the orbits~$\Gamma Ua_{\alpha}(y)$. To this end we will identify~$U\cong\mathbb{R}^{d}$ and we will examine the field~$\field$ as a subset of~$U$. For ease of notation, we define an action of~$(\mathbb{R}\setminus\{0\})^{d}$ on~$\mathbb{R}^{d}$ by coordinatewise multiplication. Given a totally positive element~$x\in \field$ and~$\beta\in\mathbb{R}$, we denote by~$x^{\beta}$ the vector obtained by taking the coordinatewise~$\beta$-th power of the image of~$x$ under the embedding defined in Section \ref{sec:setup}. We will sometimes write~$v/x=x^{-1}v$ for~$v\in\R^{d}$ and~$x\in \field$ totally positive. Using this notation, one finds that~$\Gamma Ua_{\alpha}(y)\cong\mathbb{R}^{d}/y^{2\alpha}\mathfrak{o}$ for all totally positive~$y\in\mathfrak{o}$. Note that the lattice~$y^{2\alpha}\mathfrak{o}\subseteq\R^{d}$ has covolume~$N(y)^{2\alpha}\mathrm{covol}(\mathfrak{o})$. In the special case~$\alpha=\frac{1}{2}$, the set~$y^{2\alpha}\mathfrak{o}$ is the lattice given by embedding the principal ideal~$y\mathfrak{o}$ in~$\mathbb{R}^{d}$. Consider the subset of~$\R^{d}/y\mathfrak{o}$ corresponding to the image of~$\mathfrak{o}/y\mathfrak{o}$, which is given by the embedding~$\gamma+(y)\mapsto y^{-1}\gamma+\mathfrak{o}$.
\begin{definition}
  We denote by 
  \begin{equation*}
    \Pcal_{y}^{\alpha}=\big\{\Gamma u_{y^{-1}j}a_{\alpha}(y)\mid j\in\mathfrak{o}/y\mathfrak{o}\big\}\subset \Gamma Ua_{\alpha}(y).
  \end{equation*}
  the \emph{set of rational points of denominator~$y$} on~$\Gamma Ua_{\alpha}(y)$.
\end{definition}
For what follows,~$\mathcal{F}$ denotes a symmetric fundamental parallelepiped for~$\mathfrak{o}$ in~$\mathbb{R}^{d}$. In order to prove Proposition \ref{prop:rationalpoints}, we need the following elementary analogy to the situation in~$\mathbb{R}$.
\begin{lemma}\label{lem:injectivesets}
  Let~$\alpha\in[0,1]$, ~$y\in\mathfrak{o}$ totally positive,~$\mathcal{R}_{y}$ a fixed choice of representatives for~$\mathfrak{o}/y\mathfrak{o}$, and~$\lambda\in[-1,1]$. If~$j,j'\in\mathcal{R}_{y}$ are distinct, then
  \begin{equation*}
    \big(y^{2\alpha-1}\lambda\mathcal{F}+y^{2\alpha-1}j\big)\cap\big(y^{2\alpha-1}\lambda\mathcal{F}+y^{2\alpha-1}j'\big)=\emptyset
  \end{equation*}
  In particular, the set~$y^{2\alpha-1}\lambda\mathcal{F}$ is injective on~$\mathcal{P}_{y}^{\alpha}$ in the sense of Proposition \ref{prop:formularationalpoints}. Moreover, the set
  \begin{equation*}
    \mathcal{F}_{y,\alpha}=y^{2\alpha-1}\bigg(\bigsqcup_{j\in\mathcal{R}_{y}}\mathcal{F}+j\bigg)
  \end{equation*}
  is a fundamental domain for~$\mathbb{R}^{d}/y^{2\alpha}\mathfrak{o}$.
\end{lemma}
\begin{proof}
  The action of~$(\mathbb{R}\setminus\{0\})^{d}$ on~$\mathbb{R}^{d}$ is given by multiplation with invertible diagonal matrices, hence it suffices to show that~$(\lambda\mathcal{F}+j)\cap(\lambda\mathcal{F}+j')=\emptyset$. This follows from injectivity of~$\lambda\mathcal{F}$ for~$\mathfrak{o}$ and~$j-j'\in\mathfrak{o}$. The last part of the statement is clear.
\end{proof}
The main statement of this section is the following
\begin{proposition}\label{prop:rationalpoints}
  There are~$\delta,\kappa_{2}>0$, an~$\Ltwo$-Sobolev norm~$\mathcal{S}$ on~$C_{c}^{\infty}(\Gamma\backslash G)$ depending only on the number field~$\field$ such that for all~$\alpha\in(0,\frac{1}{2}+\delta)$, all totally positive~$y\in\mathfrak{o}$, and all~$f\in C_{c}^{\infty}(\Gamma\backslash G)$
  \begin{equation*}
    \bigg\lvert\frac{1}{N(y)}\sum_{x\in\Pcal_{y}^{\alpha}}f(x)-\int_{\Gamma\backslash G}f\bigg\rvert\ll N(y)^{-\kappa_{2}\alpha}\mathcal{S}(f).
  \end{equation*}
\end{proposition}
\begin{proof}
  Using Corollaries \ref{cor:rankone} and \ref{cor:unitdecomposition}, we can assume without loss of generality that~
  \begin{equation}\label{eq:diagonalelement}
    \sigma_{i}y\asymp N(y)^{\frac{1}{d}}\qquad(1\leq i\leq d).
  \end{equation}
  For~$\alpha$ bounded away from~$\frac{1}{2}$, the proposition is a direct consequence of the mean-value theorem, similar in spirit to the approximation of the sparse subset by (part of) the full orbit in the proof of Proposition \ref{prop:formularationalpoints} in Section \ref{sec:horospheres}. Let~$0<\alpha<\frac{1}{2}-\tau$ for some~$\tau>0$. Using Proposition \ref{prop:equidistributionoflonghorospheres}, it suffices to approximate the average over the rational points to the average along~$\Gamma Ua_{\alpha}(y)$. As of Lemma \ref{lem:injectivesets}, the set
  \begin{equation*}
    \mathcal{F}_{y}=y^{-1}\bigg(\bigsqcup_{j\in\mathcal{R}_{y}}\mathcal{F}+j\bigg)
  \end{equation*}
  is a fundamental domain for~$\mathfrak{o}$, and the Haar measure on~$\Gamma Ua_{\alpha}(y)$ is given by
  \begin{equation*}
    \int_{\Gamma Ua_{\alpha}(y)}f=\frac{1}{\mathrm{vol}(\mathcal{F}_{y})}\int_{\mathcal{F}_{y}}(a_{\alpha}(y)\cdot f)(\Gamma u_{t})\der{}t,\quad\big(f\in C_{c}(\Gamma\backslash G)\big)
  \end{equation*}
  where~$\der{}t$ denotes Lebesgue measure on~$\R^{d}$. For~$f\in C_{c}^{\infty}(\Gamma\backslash G)$ follows
\begin{align*}
    \frac{1}{N(y)}&\sum_{x\in\Pcal_{y}^{\alpha}}f(x)-\int_{\Gamma Ua_{\alpha}(y)}f\\
    &=\frac{1}{N(y)}\sum_{k\in \mathcal{R}_{y}}\bigg(f\big(\Gamma u_{\frac{k}{y}}a_{\alpha}(y)\big)-\frac{N(y)}{\mathrm{vol}(\mathcal{F}_{y})}\int_{y^{-1}\mathcal{F}}f\big(\Gamma u_{t+\frac{k}{y}}a_{\alpha}(y)\big)\der{}t\bigg)\\
    &=\frac{1}{N(y)}\sum_{k\in \mathcal{R}_{y}}\frac{1}{\mathrm{vol}(\mathcal{F})}\int_{\mathcal{F}}f(\Gamma a_{\alpha}(y)u_{y^{2\alpha-1}k})-f(\Gamma a_{\alpha}(y)u_{y^{2\alpha-1}(t+k)})\der{}t.
  \end{align*}
  The exponential map~$\exp:\mathfrak{g}\to G$ is bi-Lipschitz on a neighbourhood of~$0\in\mathfrak{g}$ (cf.~\cite{volume-one}), hence denoting by~$d_{G}$ a left-invariant metric on~$G$, we obtain
  \begin{equation*}
    d_{G}(u_{t},u_{s})\ll\lVert t-s\rVert_{\infty}\qquad (t,s\in\mathbb{R}^{d}).
  \end{equation*}
  Using the mean-value theorem, inequality \eqref{eq:sobolevembedding} and assumption \eqref{eq:diagonalelement}, we find that for all~$t\in\mathcal{F}$
  \begin{align*}
    \lvert f(\Gamma a_{\alpha}(y)u_{y^{2\alpha-1}k})-f(\Gamma&a_{\alpha}(y)u_{y^{2\alpha-1}(t+k)})\rvert\ll\lVert\nabla f\rVert_{\infty}\lVert y^{2\alpha-1}t\rVert_{\infty}\\
    &\ll\mathcal{S}(f)\lvert N(y)\rvert^{\frac{2\alpha-1}{d}}\mathrm{diam}(\mathcal{F})
  \end{align*}
  for some~$\Ltwo$-Sobolev norm~$\mathcal{S}$ on~$C_{c}^{\infty}(\Gamma\backslash G)$. Plugging this bound into the preceding expression, it follows that
  \begin{equation*}
    \bigg\lvert\frac{1}{N(y)}\sum_{x\in\Pcal_{y}^{\alpha}}f(x)-\int_{\Gamma Ua_{\alpha}(y)}f\bigg\rvert\ll N(y)^{\frac{2\alpha-1}{d}}\mathcal{S}(f).
  \end{equation*}
  Now choose~$\kappa_{2}>0$ such that~$\frac{2\alpha-1}{d}\leq -\kappa_{2}\alpha$ uniformly for all~$\alpha\in(0,\frac{1}{2}-\tau]$. Note for the following that we can assume that~$\tau$ was arbitrarily small, potentially at the expense of having to choose a smaller~$\kappa_{2}$. Combining this bound with the effective equidistribution found in Proposition \ref{prop:equidistributionoflonghorospheres}, the claim follows for~$\alpha\in(0,\frac{1}{2}-\tau]$

  Assume now that~$\alpha=\frac{1}{2}$. Using the notation from Section \ref{sec:dirichlet}, assumption \eqref{eq:diagonalelement}, and Corollary \ref{cor:rankone}, we have
  \begin{equation}\label{eq:specialshape}
    a_{\alpha}(y)=\exp\big((\tfrac{1}{2d}\log N(y))H_{d}\big)g
  \end{equation}
  for some~$g$ in a fixed compact subset of~$A$. Note that~$g$ has nonnegative entries because~$y$ is totally positive. In what follows, we let~$T=\frac{1}{2d}\log N(y)$. Let~$u_{1}$ denote the diagonal embedding of~$(\begin{smallmatrix}1 & 1 \\ 0 & 1\end{smallmatrix})$ in~$G$. Then~$\mathcal{P}_{y}^{\alpha}$ is~$u_{1}$-invariant,~$u_{1}\in\Gamma$ is central in~$U$ and generates an unbounded subgroup. Using \eqref{eq:choiceofroots} we have
  \begin{equation*}
    \rho_{+}(\log a_{\alpha}(y))=\tfrac{1}{2}\log N(y)+O(1).
  \end{equation*}
  Choose the set~$F=N(y)^{-\frac{\eta}{2d}}g^{-1}\mathcal{F}g\subseteq U$ for~$\eta$ as in Proposition \ref{prop:formularationalpoints}. Then~$F$ is injective on~$\mathcal{P}_{y}^{\alpha}$. Furthermore, we have~$\mathrm{vol}(F)\asymp N(y)^{-\frac{\eta}{2}}\mathrm{vol}(\mathcal{F})$ and~$\mathrm{diam}(F)\asymp N(y)^{-\frac{\eta}{2d}}\mathrm{diam}(\mathcal{F})$. Hence Proposition \ref{prop:formularationalpoints} implies that
  \begin{equation*}
    \bigg\lvert\frac{1}{N(y)}\sum_{x\in\mathcal{P}_{y}^{\alpha}}f(x)-\int_{\Gamma\backslash G}f\bigg\rvert\ll\big(N(y)^{-\alpha\frac{\eta}{2}}+N(y)^{-\alpha\frac{\eta}{d}}\mathrm{diam}(\mathcal{F})\big)\mathcal{S}(f).
  \end{equation*}

  Finally, effective equidistribution for a given~$\alpha$ can be extended to a neighbourhood of~$\alpha$ at the expense of a slightly worse exponent. To this end let~$\delta>0$ and assume that the statement in the proposition is true for some~$\alpha\in(0,1)$ with a positive rate~$\kappa^{\prime}$ (possibly) depending on~$\alpha$. Similarly to before, one obtains from Lemma \ref{lem:shiftedsobolevnorms} that
  \begin{equation*}
    \bigg\lvert\frac{1}{N(y)}\sum_{x\in\mathcal{P}_{y}^{\alpha+\delta}}f(x)-\int_{\Gamma\backslash G}f\bigg\rvert\ll N(y)^{-\alpha\kappa^{\prime}}\mathcal{S}(a_{\delta}(y)\cdot f)\ll N(y)^{-\alpha\kappa^{\prime}+\frac{2\ell\delta}{d}}\mathcal{S}(f),
  \end{equation*}
  where~$\ell>0$ is the degree of~$\mathcal{S}$. Returning to the explicit case~$\alpha=\frac{1}{2}$, fix~$\tau>0$ such that~$1+2\delta\neq 0$ and~$\frac{d\kappa^{\prime}-4\delta\ell}{d+2d\delta}>0$ holds for all~$\lvert\delta\rvert\leq 2\tau$. Let~$\kappa^{\prime\prime}=\sup_{\lvert\delta\rvert\leq2\tau}\frac{d\kappa^{\prime}-4\delta\ell}{d+2d\delta}>0$, then it follows that the proposition holds for~$\kappa_{2}<\min\{-\tfrac{2\tau}{d},\kappa^{\prime\prime}\}$.
\end{proof}
\section{Primitive rational points on expanding horospheres}\label{sec:primitiverationalpoints}
We know from Section \ref{sec:rationalpoints} that for fixed~$\alpha\in(0,\frac{1}{2}+\delta)$ the counting measure~$\mu_{y}^{\alpha}$ on~$\Pcal_{y}^{\alpha}$ converges towards the Haar measure on~$\Gamma\backslash G$ as~$\lvert N(y)\rvert\to\infty$. From this, one deduces relatively easily that for~$\lvert N(y)\rvert\to\infty$ the primitive rational points
\begin{equation*}
  \mathcal{P}_{y}^{\alpha,\times}=\set{\Gamma u_{y^{-1}j}a_{\alpha}(y)}{j\in(\mathfrak{o}/y\mathfrak{o})^{\times}}
\end{equation*}
equidistribute towards the~$G$-invariant probability measure on~$\Gamma\backslash G$, as long as the sets~$\Pcal_{y}^{\alpha,\times}$ are sufficiently large in relation to~$\mathcal{P}_{y}^{\alpha}$, i.e.~as long as the cardinality~$\phi(y)=\big\lvert(\mathfrak{o}/y\mathfrak{o})^{\times}\big\rvert$ is not too small. To this end, denote~$\mathcal{P}_{y}^{\alpha,0}=\mathcal{P}_{y}^{\alpha}\setminus\mathcal{P}_{y}^{\alpha,\times}$, and let~$\mu_{y}^{\alpha,\times}$ and~$\mu_{y}^{\alpha,0}$ denote the normalized counting measures on~$\Pcal_{y}^{\alpha,\times}$ and~$\Pcal_{y}^{\alpha,0}$ respectively, so that
\begin{equation}\label{eq:convexcombination}
  \mu_{y}^{\alpha}=\frac{\phi(y)}{\lvert N(y)\rvert}\mu_{y}^{\alpha,\times}+\frac{\lvert N(y)\rvert-\phi(y)}{\lvert N(y)\rvert}\mu_{y}^{\alpha,0}
\end{equation}
Let~$(y_{k})_{k\in\N}$ be a sequence in~$\mathfrak{o}$ satisfying~$\lvert N(y_{k})\rvert\to\infty$ and let~$\nu$ be a weak-$\ast$ limit of~$\{\mu_{y_{k}}^{\alpha,\times};k\in\N\}$. After restricting to a subsequence, we can assume that~$\frac{\phi(y_{k})}{\lvert N(y_{k})\rvert}\to\lambda\in[0,1]$. Let~$f\in C_{c}(\Gamma\backslash G)$, then
\begin{align*}
  \int_{\Gamma\backslash G}f&=\lim_{k\to\infty}\mu_{y_{k}}^{\alpha}(f)=\lim_{k\to\infty}\left\{\frac{\phi(y_{k})}{\lvert N(y_{k})\rvert}\mu_{y_{k}}^{\alpha,\times}(f)+\frac{\lvert N(y_{k})\rvert-\phi(y_{k})}{\lvert N(y_{k})\rvert}\mu_{y_{k}}^{\alpha,0}(f)\right\}\\
  &=\lambda\nu(f)+(1-\lambda)\tilde{\nu}(f)
\end{align*}
for some measure~$\tilde{\nu}(f)$. By Dirichlet's unit theorem there is~$\varepsilon\in\mathfrak{o}_{>0}^{\times}$ such that~$a(\varepsilon)$ generates an unbounded subgroup of~$G$. Note that~$\Gamma u_{k/y}a_{\alpha}(y)a(\varepsilon)=\Gamma u_{\varepsilon^{2}k/y}a_{\alpha}(y)$, thus~$a(\varepsilon)$ preserves both~$\mu_{y_{k}}^{\alpha,\times}$ and~$\mu_{y_{k}}^{\alpha,0}$ as well as the~$G$-invariant probability measure on~$\Gamma\backslash G$, and it acts ergodically on~$\Gamma\backslash G$. Hence extremality of ergodic measures implies that~$\nu$ is the~$G$-invariant probability measure if~$\lambda>0$.

In what follows, we give a proof without the assumption~$\lambda>0$ by deriving an effective version, i.e.~we provide a bound of the form~$\lvert\mu_{y}^{\alpha,\times}(f)-E_{f}\rvert\ll\psi(y)S(f)$ for some~$\psi:\mathfrak{o}\to\R_{>0}$ decaying in~$N(y)$. To this end we once again adapt the reasoning from~\cite{Akka2016}.
\begin{remark}
  We want to point out that up to now we never really used that the degree~$d=[\field,\Q]$ was larger than one. The upcoming argument is the only place where it is needed. In particular, as mentioned at the end of Section~\ref{sec:horospheres}, one can deduce equidistribution of rational points on horospheres in~$\SL_{2}(\Z)\backslash\SL_{2}(\R)$ in pretty much the same way. In order to select the primitive rational points, we are going to use the assumption~$d>1$ to obtain an element acting with a spectral gap and preserving the set of primitive rational points as in the sketch before. This step is a bit trickier in the situation where~$d=1$, and involves lifting the problem to a~$p$-adic cover. This is examined in~\cite{sl2}.
\end{remark}
Throughout this section we fix a unit~$\varepsilon\in\mathfrak{o}_{>0}^{\times}$ generating a non-compact subgroup. Given~$K\in\N$, and~$f\in C_{c}^{\infty}(\Gamma\backslash G)$ real-valued, let
\begin{equation*}
  D_{K}f(x)=\frac{1}{K}\sum_{k=0}^{K-1}f\big(xa(\varepsilon)^{k})-E_{f}
\end{equation*}
be the discrepancy (for~$a(\varepsilon)$) of~$f$.
\begin{lemma}\label{lem:subcollectionstrick}
  There exists an~$\Ltwo$-Sobolev norm~$\mathcal{S}$ on~$C_{c}^{\infty}(\Gamma\backslash G)$ and some~$\varepsilon\in\mathfrak{o}_{>0}^{\times}$ such that for all real-valued~$f\in C_{c}^{\infty}(\Gamma\backslash G)$
  \begin{equation*}
    \int_{\Gamma\backslash G}(D_{K}f)^{2}\ll K^{-1}\mathcal{S}(f)^{2}.
  \end{equation*}
\end{lemma}
\begin{proof}
  In order to use \eqref{eq:matrixcoefficients}, note that~$G=KA_{+}K$, where~$K=\SO(2)^{d}$ and for any decomposition~$g=kg_{+}l$ ($k,l\in K$,~$g_{+}\in A_{+}$), the element~$g_{+}$ is uniquely determined by~$g$. It is well-known that the Harish-Chandra spherical function~$\Xi$ is bi-invariant under~$K$, i.e.~$\Xi(g)=\Xi(g_{+})$ (cf.~\cite[Ch.~7,~\textsection8]{Knapp1986}). Note that for~$\ell\in\Z$
  \begin{equation*}
    \varrho_{+}\big(\log a(\varepsilon^{\ell})_{+}\big)=\sum_{i=1}^{d}\lvert\ell\log\sigma_{i}(\varepsilon)\rvert\geq\lvert\ell\rvert\log\lVert a(\varepsilon)\rVert_{\infty}.
  \end{equation*}
  Hence \eqref{eq:matrixcoefficients} implies that for all~$f_{1},f_{2}\in C_{c}^{\infty}(\Gamma\backslash G)$
  \begin{equation}\label{eq:matrixcoefficientsnondiagonal}
    \lvert\langle a(\varepsilon^{\ell})f_{1},f_{2}\rangle-E_{f_{1}}\overline{E_{f_{2}}}\rvert\ll\lVert a(\varepsilon)\rVert_{\infty}^{-\kappa_{\tau}\ell}\mathcal{S}(f_{1})\mathcal{S}(f_{2}).
  \end{equation}
  We remark, that Proposition \ref{prop:effectivevonneumann} immediately implies that
  \begin{equation*}
    \int_{\Gamma\backslash G}(D_{K}f)^{2}\ll\big(K^{-\varsigma}+\lVert a(\varepsilon)\rVert_{\infty}^{-\kappa_{\tau}K^{1-\varsigma}}\big)\mathcal{S}(f),
  \end{equation*}
  which is too weak for what we want. Using Dirichlet's unit theorem, we can without loss of generality assume that~$\lVert a(\varepsilon)\rVert_{\infty}^{-\kappa_{\tau}}\geq 2$. Application of \eqref{eq:sobolevembedding} and \eqref{eq:matrixcoefficientsnondiagonal} as in the proof of Proposition \ref{prop:effectivevonneumann} yields
  \begin{align*}
    \int_{\Gamma\backslash G}(D_{K}f)^{2}&=\frac{1}{K^{2}}\bigg(K\lVert f-E_{f}\rVert_{2}^{2}+2\sum_{0\leq j<\ell<K}\langle a(\varepsilon)^{\ell-j}(f-E_{f}),(f-E_{f})\rangle\bigg)\\
    &\ll\frac{\mathcal{S}(f)^{2}}{K}+\frac{\mathcal{S}(f)^{2}}{K^{2}}\sum_{0\leq j<\ell<K}\lVert a(\varepsilon)\rVert_{\infty}^{-\kappa_{\tau}(\ell-j)}\\
    &\ll\frac{\mathcal{S}(f)^{2}}{K}+\frac{\mathcal{S}(f)^{2}}{K^{2}}\frac{\lVert a(\varepsilon)\rVert_{\infty}^{-\kappa_{\tau}(K+1)}}{(\lVert a(\varepsilon)\rVert_{\infty}^{-\kappa_{\tau}}-1)^{2}}\ll K^{-1}\mathcal{S}(f)^{2}.
  \end{align*}
\end{proof}
\begin{lemma}\label{lem:propertiessobolev}
  Given an~$\Ltwo$-Sobolev norm~$\mathcal{S}$ on~$C_{c}^{\infty}(\Gamma\backslash G)$, there exist~$\beta>0$ and an~$\Ltwo$-Sobolev norm~$\mathcal{S}_{2}$ such that for any real-valued~$f\in C_{c}^{\infty}(\Gamma\backslash G)$ and~$K\in\N$
  \begin{equation*}
    \mathcal{S}\big((D_{K}f)^{2}\big)\leq e^{\beta K}\mathcal{S}_{2}(f)^{2}.
  \end{equation*}
\end{lemma}
\begin{proof}
  As of Lemma \ref{lem:productofsobolevnorms}, we have
  \begin{equation*}
    \mathcal{S}\big((D_{K}f)^{2}\big)\ll\mathcal{S}_{2}(D_{K}f)^{2}.
  \end{equation*}
  Using Lemma \ref{lem:shiftedsobolevnorms}, we get
  \begin{equation*}
    \mathcal{S}_{2}(D_{K}f)\leq\frac{1}{K}\sum_{k=0}^{K-1}\mathcal{S}_{2}\big(a(\varepsilon)^{-k}\cdot f\big)\ll\frac{\mathcal{S}_{2}(f)}{K}\sum_{k=0}^{K-1}\lVert a(\varepsilon)\rVert_{\infty}^{k\iota}
  \end{equation*}
  for some~$\iota>0$. Thus~$\beta=\frac{\iota}{2}\log\lVert a(\varepsilon)\rVert_{\infty}$ will do.
\end{proof}

We are finally in the position to prove the main theorem.
\begin{proof}[Proof of Theorem \ref{thm:mainthmnumberfield}]
  Let~$y\in\mathfrak{o}$ be totally positive and assume that~$f\in C_{c}^{\infty}(\Gamma\backslash G)$ is real-valued. Using \eqref{eq:convexcombination}, one obtains
\begin{align*}
  \notag\frac{\phi(y)}{N(y)}\lvert\mu_{y}^{\alpha,\times}(f-E_{f})\rvert^{2}&=\frac{\phi(y)}{N(y)}\lvert\mu_{y}^{\alpha,\times}(D_{K}f)\rvert^{2}\leq\frac{\phi(y)}{N(y)}\mu_{y}^{\alpha,\times}\big((D_{K}f)^{2}\big)\\
  &\leq\mu_{y}^{\alpha}\big((D_{K}f)^{2}\big).
\end{align*}
The function~$(D_{K}f)^{2}-{E_{f}}^{2}$ is smooth with compact support. Using Proposition \ref{prop:rationalpoints}, we obtain
\begin{align}
  \label{eq:bounddiscrepancyrationalpoints}\mu_{y}^{\alpha}\big((D_{K}f)^{2}\big)&\ll\int_{\Gamma\backslash G}(D_{K}f)^{2}+N(y)^{-\kappa_{2}\alpha}\mathcal{S}\big((D_{K}f)^{2}-{E_{f}}^{2}\big)\\
  \notag&\ll\int_{\Gamma\backslash G}(D_{K}f)^{2}+N(y)^{-\kappa_{2}\alpha}\big(\mathcal{S}(D_{K}f)^{2}+\mathcal{S}(f)^{2}\big),
\end{align}
where the second inequality follows from \eqref{eq:sobolevembedding} and Lemma \ref{lem:productofsobolevnorms}. As~$\varepsilon$ could be chosen arbitrary, we assume that~$\kappa_{\tau}\log\lVert a(\varepsilon)\rVert_{\infty}>1$, so that the error term in Lemma \ref{lem:subcollectionstrick} is dominated by~$K^{-1}$. Furthermore, as of Lemma \ref{lem:propertiessobolev} we can find an~$\Ltwo$-Sobolev norm~$\mathcal{S}_{2}$ dominating~$\mathcal{S}$ such that~$\mathcal{S}\big((D_{K}f)^{2}\big)\leq e^{\beta K}\mathcal{S}_{2}(f)^{2}$ for some~$\beta>0$. Applying these bounds to \eqref{eq:bounddiscrepancyrationalpoints} and multiplying the resulting expression by~$\frac{N(y)}{\phi(y)}$, we obtain
\begin{equation}
  \label{eq:prefinal}\lvert\mu_{y}^{\alpha,\times}(f-{E_{f}})\rvert^{2}\ll\tfrac{N(y)}{\phi(y)}\left(K^{-1}+N(y)^{-\kappa_{2}\alpha}e^{\beta K}\right)\mathcal{S}_{2}(f)^{2}
\end{equation}
For sufficiently large~$N(y)$ we can find an integer~$K=\delta\frac{\kappa_{2}}{\beta}\alpha\log N(y)$ with~$\delta\in(\frac{1}{2},\frac{3}{4})$. Let~$\kappa_{3}=\frac{\kappa_{2}}{4}$, so that~$N(y)^{-\kappa_{2}\alpha}e^{\beta K}\leq N(y)^{-\kappa_{3}\alpha}$ and \eqref{eq:prefinal} is bounded by
\begin{equation*}
  \lvert\mu_{y}^{\alpha,\times}(f-{E_{f}})\rvert^{2}\ll\tfrac{N(y)}{\phi(y)}\left(\tfrac{\beta}{\kappa_{2}\alpha\log N(y)}+N(y)^{-\kappa_{3}\alpha}\right)\mathcal{S}(f)^{2}
\end{equation*}
Note that~$\log N(y)\leq N(y)^{\kappa_{3}\alpha}$ for all totally positive~$y\in\mathfrak{o}$, and therefore applying the bound obtained to both the real and the imaginary part separately, we can conclude that for all~$f\in C_{c}^{\infty}(\Gamma\backslash G)$
\begin{equation}\label{eq:effectiveresult}
  \bigg\lvert\mu_{y}^{\alpha,\times}(f)-\int_{\Gamma\backslash G}f\bigg\rvert\ll\alpha^{-\frac{1}{2}}\left(\frac{N(y)}{\phi(y)\log N(y)}\right)^{\frac{1}{2}}\mathcal{S}(f)
\end{equation}
Note that the implicit constant is proportional to~$\kappa_{3}^{-1}$ and does grows as the spectral gap becomes smaller.

We can now apply Proposition \ref{prop:totientbound} to obtain the theorem.
\end{proof}
\def\cprime{$'$} \providecommand{\bysame}{\leavevmode\hbox
  to3em{\hrulefill}\thinspace}
\providecommand{\MR}{\relax\ifhmode\unskip\space\fi MR }
\providecommand{\MRhref}[2]{%
  \href{http://www.ams.org/mathscinet-getitem?mr=#1}{#2} }
\providecommand{\href}[2]{#2}


\begin{thebibliography}{10}
\bibitem[AE16]{Akka2016} M.~Aka and M.~Einsiedler.
  \newblock Duke's theorem for subcollections.
  \newblock {\em Ergodic Theory Dynam. Systems} 36 (2016), no.~2, 335--342.
  
\bibitem[Be98]{Bekka} M.~B.~Bekka.
  \newblock On uniqueness of invariant means.
  \newblock {\em Proc.~Amer.~Math.~Soc.~}126 (1998), no.~2, 507--514. 
  
\bibitem[B91]{Borel} A.~Borel.
  \newblock Linear algebraic groups.
  \newblock Second edition. Graduate Texts in Mathematics, 126. Springer-Verlag, New York, 1991.

\bibitem[Bo06]{Bourbaki} N.~Bourbaki.
  \newblock El\'ments de math\'ematique. Fasc.~XXXVIII: Groupes et alg\`ebres de Lie. Chapitre VII: Sous-alg\`ebres de Cartan, \'el\'ements r\'eguliers. Chapitre VIII: Alg\`ebres de Lie semi-simples d\'eploy\'ees.
  \newblock Springer-Verlag, 2006.

\bibitem[BS91]{BurgerSarnak} M.~Burger and P.~Sarnak.
  \newblock Ramanujan Duals II.
  \newblock {\em Invent. Math.} 106 (1991), no.~1, 1--11.
  

\bibitem[CHH88]{Cowling1988} M.~Cowling, U.~Haagerup, and R.~Howe.
  \newblock Almost $L^{2}$ matrix coefficients.
  \newblock {\em J. Reine Angew. Math.} 387 (1988),97--110.
  
\bibitem[ELu18]{survey} M.~Einsiedler and M.~Luethi.
  \newblock Kloosterman Sums, Disjointness, and Equidistribution. {\em Ergodic theory and dynamical systems in their interactions with arithmetics and combinatorics}, 137–-161,
  \newblock {\em Lecture Notes in Math.} 2213, {\em Springer}, {\em Cham}, 2018.

\bibitem[ELuS19]{sl2} M.~Einsiedler, M.~Luethi, and N.~Shah.
  \newblock Primitive rational points on expanding horocycles in products of the modular surface with the torus.
  \newblock Preprint, arXiv:1901.03078v3.

\bibitem[EMV09]{EMV} M.~Einsiedler, G.~Margulis, and A.~Venkatesh.
  \newblock Effective equidistribution for closed orbits of semisimple groups on homogeneous spaces.
  \newblock {\em Invent. Math.} 177 (2009), no.~1, 137–-212.

\bibitem[EMSS16]{Primitive} M.~Einsiedler, S.~Mozes, N.~Shah, and U.~Shapira.
  \newblock Equidistribution of primitive rational points on expanding horospheres.
  \newblock {\em Compos. Math.} 152 (2016), no.~4, 667–-692.

\bibitem[EW11]{volume-one} M.~Einsiedler and T.~Ward.
  \newblock Ergodic theory with a view towards number theory.
  \newblock {\em Graduate Texts in Mathematics} 259. {\em Springer-Verlag London, Ltd.}, {\em London}, 2011.

\bibitem[EM93]{EskinMcMullen} A.~Eskin and C.~McMullen.
  \newblock Mixing, counting, and equidistribution in Lie groups.
  \newblock {\em Duke Math. J.} 71 (1993), no.~1, 181–-209. 
  
\bibitem[HW08]{HardyWright} G.~Hardy and E.~Wright.
  \newblock An introduction to the theory of numbers. Sixth edition. Revised by D. R. Heath-Brown and J. H. Silverman. With a foreword by Andrew Wiles.
  \newblock {\em Oxford University Press}, {\em Oxford}, 2008.

\bibitem[KM96]{KleinbockMargulisPieces} D.~Kleinbock and G.~Margulis.
  \newblock Bounded orbits of nonquasiunipotent flows on homogeneous spaces. {\em Sina\u\i's Moscow Seminar on Dynamical Systems}, 141–-172, Amer. Math. Soc. Transl. Ser. 2, 171, Adv. Math. Sci., 28, {\em Amer. Math. Soc.}, {\em Providence, RI}, 1996.

\bibitem[Kn86]{Knapp1986} A.~Knapp.
  \newblock Representation theory of semisimple groups. An overview based on examples.
  \newblock Princeton Mathematical Series, 36. {\em Princeton University Press}, {\em Princeton, NJ}, 1986.

\bibitem[M10]{Marklof2010} J.~Marklof.
  \newblock The asymptotic distribution of Frobenius numbers.
  \newblock {\em Invent. Math.} 181 (2010), no.~1, 179-–207.

\bibitem[Ne92]{Neukirch} J.~Neukirch.
  \newblock Algebraische Zahlentheorie.
  \newblock {\em Springer-Verlag, Berlin}, 1992.

\bibitem[Sar81]{Sarnak1981} P.~Sarnak.
  \newblock Asymptotic behavior of periodic orbits of the horocycle flow and Eisenstein series.
  \newblock {\em Comm. Pure Appl. Math.} 34 (1981), no.~6, 719–-739. 

\bibitem[Sel65]{Selberg65} A.~Selberg.
  \newblock On the estimation of Fourier coefficients of modular forms.
  \newblock 1965 {\em Proc. Sympos. Pure Math., Vol. VIII} pp. 1–15 {\em Amer. Math. Soc.}, {\em Providence, RI}

\bibitem[Sp09]{SpringerLinearGroups} T.~Springer.
  \newblock Linear algebraic groups. Reprint of the 1998 second edition.
  \newblock Modern Birkhäuser Classics. {\em Birkhäuser Boston, Inc.}, {\em Boston, MA}, 2009.

\bibitem[To02]{Tomanov2002} G.~Tomanov.
  \newblock Actions of maximal tori on homogeneous spaces. {\em Rigidity in dynamics and geometry (Cambridge, 2000)}, 407-–424, {\em Springer, Berlin}, 2002.

\bibitem[Ve10]{Venkatesh2010} A.~Venkatesh.
  \newblock Sparse equidistribution problems, period bounds and subconvexity.
  \newblock {\em Ann. of Math. (2)} 172 (2010), no.~2, 989–-1094.
\end{thebibliography}
\end{document}